\documentclass[11pt]{amsart}
\usepackage[english]{babel}
\usepackage{amsthm,amsmath,graphicx}
\usepackage{amssymb,bbm,esint}
\numberwithin{equation}{section}
\usepackage{appendix}
\usepackage{geometry}

\usepackage[dvipsnames]{xcolor}
\usepackage[colorlinks=true,allcolors=ForestGreen]{hyperref}
\usepackage{amsthm}
\DeclareMathOperator*{\esssup}{ess\,sup}
\DeclareMathOperator*{\essinf}{ess\,inf}
\DeclareMathOperator*{\essosc}{ess\,osc}

\newcommand{\intb}[2][m,r]{\int_{{#1}_{#2}}}
\newcommand{\gaglifract}[1][u]{\frac{|{#1}(x,t)-{#1}(y,t)|^{p}}{|y^{-1}\circ x|^{Q+sp}_{\mathbb{H}^N}}}
\newcommand{\gaglifrac}[1][u]{\frac{|{#1}(x)-{#1}(y)|^{p}}{|y^{-1}\circ x|^{Q+sp}_{\mathbb{H}^N}}}
\newcommand{\tail}[5][u,x,y,r,B]{\int_{\mathbb{H}^N\setminus{{#5}_{#4}}}\frac{{#1}^{p-1}({#3},t)}{|{#3}^{-1}\circ{#2}|^{Q+sp}_{\mathbb{H}^N}}}
\newcommand{\dist}[2][x,y]{|{#2}^{-1}\circ {#1}|_{\mathbb{H}^N}}
\newcommand{\hn}{\mathbb{H}^N}

\def\YYint#1#2#3{{\setbox0=\hbox{$#1{#2#3}{\iint}$}
		\vcenter{\hbox{$#2#3$}}\kern-.51\wd0}}
 
\usepackage[capitalize,nameinlink]{cleveref}
\crefname{section}{Section}{Sections}
\crefname{subsection}{Subsection}{Subsections}
\crefname{subsection}{subsection}{subsections}
\crefname{condition}{Condition}{Conditions}
\crefname{hypothesis}{Hypothesis}{Conditions}
\crefname{assumption}{Assumption}{Assumptions}
\crefname{lemma}{Lemma}{Lemmas}
\crefname{claim}{Claim}{Claims}
\crefname{observation}{Observation}{Observations}
\crefname{example}{Example}{Examples}
\crefname{remark}{Remark}{Remarks}

\crefformat{equation}{\textup{#2(#1)#3}}
\crefrangeformat{equation}{\textup{#3(#1)#4--#5(#2)#6}}
\crefmultiformat{equation}{\textup{#2(#1)#3}}{ and \textup{#2(#1)#3}}
{, \textup{#2(#1)#3}}{, and \textup{#2(#1)#3}}
\crefrangemultiformat{equation}{\textup{#3(#1)#4--#5(#2)#6}}%
{ and \textup{#3(#1)#4--#5(#2)#6}}{, \textup{#3(#1)#4--#5(#2)#6}}%
{, and \textup{#3(#1)#4--#5(#2)#6}}

\Crefformat{equation}{#2Equation~\textup{(#1)}#3}
\Crefrangeformat{equation}{Equations~\textup{#3(#1)#4--#5(#2)#6}}
\Crefmultiformat{equation}{Equations~\textup{#2(#1)#3}}{ and \textup{#2(#1)#3}}
{, \textup{#2(#1)#3}}{, and \textup{#2(#1)#3}}
\Crefrangemultiformat{equation}{Equations~\textup{#3(#1)#4--#5(#2)#6}}%
{ and \textup{#3(#1)#4--#5(#2)#6}}{, \textup{#3(#1)#4--#5(#2)#6}}%
{, and \textup{#3(#1)#4--#5(#2)#6}}

\crefdefaultlabelformat{#2\textup{#1}#3}
\newtheorem{theorem}{Theorem}[section]
\newtheorem{lemma}[theorem]{Lemma}

\newtheorem{proposition}[theorem]{Proposition}

\newtheorem{remark}[theorem]{Remark}        

\numberwithin{equation}{section}

\title[Holder regularity for parabolic equations in the Heisenberg group]{{Optimal in Tail H\"older Estimates for weak solutions of the nonlocal parabolic $\MakeLowercase{p}$-Laplace equations on the Heisenberg Group}}
\begin{document}
\author[Debraj Kar]{Debraj Kar}
\address{Department of Mathematics,  University of Kalyani}
\email{ prometheus.math98@gmail.com, debrajmath22@klyuniv.ac.in}
\begin{abstract}
  We prove the H\"older continuity for weak solutions to parabolic p-Laplace equations on the Heisenberg group. We deduce this result while considering an optimal tail condition.  
\end{abstract}
\maketitle

\setcounter{secnumdepth}{1}
\setcounter{tocdepth}{1}
\tableofcontents
\section{Introduction}
The goal of this paper is to prove H\"older regularity of weak solutions of 
\begin{align}\tag{main}\label{main}
    \partial_tu(x,t)+\mbox{p.v. }\int_{\mathbb{H}^N}K(x,y,t)|u(x,t)-u(y,t)|^{p-2}(u(x,t)-u(y,t))dy=0 ,
\end{align}
on $\Omega_T:=\Omega\times(0,T]$ where $\Omega$ is a bounded open subset in the Heisenberg Group $\mathbb{H}^N$, $T>0$ and
p.v. denotes the Cauchy principal value and the symmetric kernel $K:\mathbb{H}^N\times\mathbb{H}^N\times\mathbb{R}\rightarrow\mathbb{R}$ satisfies
\begin{align}\label{eq1.1}
    \frac{\lambda'}{\dist[x]{y}^{Q+sp}}\leq K(x,y,t)\leq \frac{\Lambda'}{\dist[x]{y}^{Q+sp}}
\end{align}
for some $\lambda',\Lambda'>0$ a.e. and $(x,y,t)\in\mathbb{H}^N\times\mathbb{H}^N\times\mathbb{R}$.
\subsection{Notations, Definitions and Main Result}
\subsubsection{Notations:} 
We introduce some notations which will be used throughout the paper: 
\begin{itemize}
    \item We shall denote $(x,t)$ to be a point in $\hn\times (0,T]$.
    \item We shall denote $B_r(x):= \{y\in \hn: \dist[x]{y}<r\}$ as ball of center $x$ with radius $r$. For simplicity, we use $B_r$ to denote the same instead.  For the generic ball, we use $B$. 
    \item Throughout this paper, we use \texttt{data} to represent the constant terms $s,p,Q,\lambda'$ and $\Lambda'$.
    \item We shall frequently use $f_t$, $\partial_tf$, $\frac{\partial f}{\partial  t}$ to indicate the time derivative of $f$.
    \item Let $z_0:=(x_0,t_0)\in \hn\times (0,T]$. Let us denote {\em forward} and {\em backward} parabolic cylinder as follows
    \begin{align*}
        \begin{cases}
            & z_0+\mathcal{Q}^\oplus_r(\theta)=B_r(x_0)\times (t_0,t_0+\theta r^{sp}]\\
            & z_0+\mathcal{Q}^\ominus_r(\theta)=B_r(x_0)\times (t_0-\theta r^{sp},t_0]
        \end{cases}
    \end{align*}
    \item We denote 
    \begin{align*}
        a_+:=\max\{a,0\},\,a_-:=-\min\{a,0\}\mbox{ for }a\in\mathbb{R}.
    \end{align*}
\end{itemize}
\subsubsection{Definitions:} Here we want to provide some definitions. Let $\Omega$ be an open subset of $\hn$ and $I$ be a subset of $\mathbb{R}$. 
\begin{itemize}
    \item The {\em fractional Sobolev space} on the Heisenberg group, denoted by $W^{s,p}(\Omega)$, is defined by 
    \begin{align*}
        W^{s,p}(\Omega):=\Bigg\{v\in L^p(\Omega): \iint_{\Omega\times\Omega}\frac{|v(x)-v(y)|^p}{\dist[x]{y}^{Q+sp}}\,dx\,dy<\infty\Bigg\}.
    \end{align*}
    We denote $W^{s,p}_0(\Omega)$ as closure of $C_0^\infty(\Omega)$ in $W^{s,p}(\Omega)$.
    \item The {\em parabolic tail space}, denoted by $L^m(I;L^m_\gamma(\Omega))$, on the Heisenberg group is defined by 
    $$L^m(I;L^m_\gamma(\Omega)):=\Bigg\{v\in L^m(I;L^m_\textup{loc}(\Omega)):\iint_{I\times\hn}\frac{|v(x,t)|^m}{1+|x|_{\hn}^{Q+\gamma}}\,dx\,dt<\infty\Bigg\},$$ for some $m,\gamma>0$.
    \item {\em Nonlocal parabolic tail : } For $z_0=(x_0,t_0)\in \hn\times I$ and $r,\tau>0$ with $(t_0-\tau,t_0)\subset I$, the nonlocal parabolic tail is defined by 
    $$\textup{ Tail }(v;z_0,r,\tau):=\Bigg(\fint_{t_0-\tau}^{t_0}r^{sp}\int_{\hn\setminus B_r(x_0)}\frac{|v(x,t)|^{p-1}}{\dist[x]{y}^{Q+sp}}\,dx\,dt\Bigg)^\frac{1}{p-1}$$
    Note that if $v\in L^{p-1}(I;L^{p-1}_{sp}(\hn))$ then it holds that $\textup{Tail }(v;z_0,r,\tau)<\infty$. In the rest of the paper, we mainly use tail on a parabolic cylinder instead, that is $$\textup{Tail}(v;z_0+\mathcal{Q}_r^\ominus(\theta)):=\int_{t_0-\theta r^{sp}}^{t_0}\int_{\hn\setminus B_r(x_0)}\frac{|v(x,t)|^{p-1}}{\dist[x]{y}^{Q+sp}}\,dx\,dt$$
    \item {\em Weak Solution : } Let $\Omega\, (\subset\hn)$ be a bounded open set and $\Omega_T=\Omega\times(0,T]$. A function $u:\hn\times(0,T]\rightarrow\mathbb{R}$ satisfying
    $$u\in L^p_{\textup{loc}}(0,T;W^{s,p}_{\textup{loc}}(\Omega)) \cap L^{p-1}_{\textup{loc}}(0,T;L^{p-1}_{sp}(\hn))\cap C_{\textup{loc}}(0,T;L^2_{\textup{loc}}(\Omega))$$ is called weak sub(super)solution of the (\ref{main}), if for every compact set $K\subset\Omega$, subinterval $[T_1,T_2]\subset (0,T]$ and for all nonnegative testing functions $$\varphi \in W^{1,2}_{\textup{loc}}(0,T;L^2(K))\cap L^p_{\textup{loc}}(0,T;W^{s,p}_0(K))$$ we have
    $$\int_K u\,\varphi\,d\eta\Biggr\vert_{T_1}^{T_2}-\int_{T_1}^{T_2}\int_K u\partial_t\varphi\,d\eta\,dt$$$$\quad+\int_{T_1}^{T_2}\iint_{\hn\times\hn}|u(x,t)-u(y,t)|^{p-2}(u(x,t)-u(y,t))(\varphi(x,t)-\varphi(y,t))\, K(x,y,t)d\xi\,d\eta\,dt\geq (\leq)\,0.$$
\end{itemize}
\subsubsection{Main Result : }
\begin{theorem}\label{T1.1}
    Let $p>1,T>0$ and $s\in(0,1)$. Let $u$ be locally bounded weak solution to (\ref{main}) in $\Omega_T$, where $\Omega$ is a bounded open subset of $\hn$. Moreover, assume that for some $\varepsilon>0$,
    $$\int_{\hn}\frac{|u(x,\cdot)|^{p-1}}{1+|x|_{\hn}^{Q+sp}}dx\in L^{1+\varepsilon}_{\textup{loc}}(0,T].$$
    Then $u$ is locally H\"older continuous in $\Omega_T$. More precisely, there exists some constants $\gamma,\tilde{\gamma},\lambda$ and $\eta\in (0,1)$ depending upon \textup{\texttt{data}}, such that for any $0<r<R<\mathcal{R}$ with $(x_0,t_0)+\mathcal{Q}_R^\ominus(\omega^{2-p})\subset(x_0,t_0)+\mathcal{Q}_\mathcal{R}^\ominus$, following holds
    \begin{align*}
        \essosc_{(x_0,t_0)+\mathcal{Q}^\ominus_r(\omega^{2-p})}\leq \gamma\omega\Big(\frac{r}{R}\Big)^\beta,
    \end{align*}
    for some $\beta:=\min\Big\{\frac{sp\textup{ ln}(1-\eta)}{\textup{ ln}((1-\eta)^{p-2}\lambda^{sp})},\frac{sp\textup{ ln }(1-\eta)}{\textup{ ln}(\bar{\gamma}^{2-p}\lambda^{sp})},\frac{\varepsilon sp}{2(1+\varepsilon)}\Big\}$ depending only on \textup{\texttt{data}}, $\varepsilon$ and 
    \begin{align*}
        \omega=2\esssup_{(x_0,t_0)+\mathcal{Q}_\mathcal{R}^\ominus}|u|+\Bigg(\fint_{t_0-\mathcal{R}^{sp}}^{t_0}\Big(\mathcal{R}^{sp} \intb[\hn\setminus B]{\mathcal{R}}\frac{|u(x,t)|^{p-1}}{|x_0^{-1}\circ x|_{\hn}^{Q+sp}}dx\Big)^{1+\varepsilon}dt\Bigg)^\frac{1}{1+\varepsilon}
    \end{align*}
\end{theorem} 
\begin{remark}
    In the paper \cite{KT25_1}, it was proved that the weak solution $u$ to (\ref{main}) is locally bounded for $p>\frac{2Q}{Q+2s}$. For $1<p\leq \frac{2Q}{Q+2s}$, it is expected that the local weak solutions are locally bounded when an additional integrability condition is imposed. In this paper, we already consider locally bounded weak solutions in our result \autoref{T1.1} since the boundary of $\frac{2Q}{Q+2s}$ does not play a role in H\"older continuity.
\end{remark}
\subsection{Background and Novelty} In the last couple of decades, integro-differential operators have drawn great interest, for instance see \cite{CT04_7,Las02_8,Ju05_9}. It  emphasizes the L\'evy process which specifies the emergence of the {\em jump diffusion}. As far as regularity of this problem is concerned, H\"older regularity was proved in the Euclidean elliptic case in the paper \cite{CKP16_10}. For $p=2$, this type of work is carried out by Silvestre \cite{Sil06_11}. For fractional p-laplace operator, one can see the H\"older estimates type of works in \cite{Coz17_6,BLS21_14,APT23_15,CDI25_19} and for other type of equations, see \cite{GK22_16,DP19_17,GL23_18}.

The De Giorgi-Nash-Moser regularity theory of the parabolic counterpart of this equation was addressed in \cite{CCV11_12,FK13_21, APT22_22,Lia24_2,Lia24_5,BK24_20}. When the kernel $K$ in (\ref{eq1.1}) is nonsymmetric, Kassmann and Weidner \cite{KW25_23} prove the H\"older estimate for the parabolic nonlocal equation.

This work is a continuation of our work in \cite{KT25_1} where we proved local boundedness for the nonlocal parabolic $p$-Laplace equation on the Heisenberg group under optimal tail conditions. The main ingredient of our approach is the energy estimate (see Lemma \ref{L2.1}). Unlike in the papers of Ding, Zhang, Zhou \cite{DZZ21}
 and Adimurthi, Prasad, Tewary \cite{APT22_22} but like \cite{Lia24_5}, we abstain to use logarithmic lemma and exponential change of variable method due to the presence of {\em good term} in \ref{eq2.1} and {\em expansion of positivity} (see Section \ref{S4}) respectively. Additionally, to manage the long range behavior of the solution, we rely on the assumption of the stronger {\em tail} condition. Previously, $L^\infty_{\textup{loc}}$ tail condition was required for time variable (see \cite{APT22_22,Lia24_2}). This changed when Kassmann and Weidner \cite{KW24_25} considered the optimal tail condition that is $L^{1+\varepsilon}_{\textup{loc}}$ tail condition in time variable to prove H\"older estimate for nonlocal heat equations. Later, Liao \cite{Lia24_5} proved H\"older for nonlocal parabolic p-Laplace equation under this optimal condition in the Euclidean setting. For the Heisenberg group, Manfredini, Palatucci, Piccinini and Polidoro \cite{MPPP23_3} proved the H\"older continuity for nonlocal elliptic equation. Other recent works on regularity theory of $p$-Laplace type equations on the Heisenberg group are \cite{PP2022,PP2024,FZ2024,FZZ2024,KT25_1,ZN2026}.
 
 In the context of nonlocal parabolic equation on the Heisenberg group, our result is new and we employ the optimal condition on tail to prove our main result. 
 \subsection{Notes on Heisenberg Group} The results in this subsection is based on the book \cite{BLU07_26}. The Heisenberg group is a nilpotent Lie Group.  This is the set $\mathbb{R}^{2N+1}$ equipped with the following group operation
 $$x\circ x'=(\xi+\xi',\eta+\eta',\Upsilon+\Upsilon'+2(\big<\xi',\eta\big>-\big<\xi,\eta'\big>),$$
 where $$x=(z,\Upsilon)=(\xi,\eta,\Upsilon)=(\xi_1,\xi_2,...,\xi_N,\eta_1,\eta_2,...,\eta_N,\Upsilon)$$
 and $$x'=(z',\Upsilon')=(\xi',\eta',\Upsilon')=(\xi'_1,\xi'_2,...,\xi'_N,\eta'_1,\eta'_2,...,\eta'_N,\Upsilon').$$
 A one parameter group of automorphism on $\hn$ is defined as 
 $$\Phi_\lambda(\xi,\eta,\Upsilon)=(\lambda \xi,\lambda \eta,\lambda^2\Upsilon),$$
 so that it has a {\em homogeneous dimension} of $Q=2N+2$. Also a {\em homogeneous norm} $d_0$ such that $d_0:\hn\rightarrow [0,\infty]$ is a function satisfying
 \begin{itemize}
     \item $d_0(\Phi_\lambda(x))=\lambda d_0(x)$ for any $x\in\hn$ and $\lambda>0$.
     \item $d_0(x)=0$ if and only if $x=0$. 
 \end{itemize}
 Define standard {\em homogeneous norm} on $\hn$ by
 $$|x|_{\hn}=(|z|^4+|\Upsilon|^2)^{1/4}.$$
 It is well known that all homogeneous norm on $\hn$ are equivalent (see \cite[Corollary 5.1.5]{BLU07_26}). Additionally, Heisenberg group with any homogeneous norm $d_0$ satisfies a pseudo-triangle inequality as in the following lemma. 
 \begin{lemma}\cite[Proposition 5.1.7]{BLU07_26}
     Consider Heisenberg group $\hn$ with homogeneous norm $d_0$. Then there exists a constant $c>0$ such that for all $x,y\in\hn$, following holds
     \begin{itemize}
         \item $d_0(x\circ y)\geq \frac{1}{c}d_0(x)-d_0(y^{-1}),$
         \item $d_0(x\circ y)\leq c(d_0(x)+d_0(y)),$
         \item $d_0(x\circ y)\geq \frac{1}{c}d_0(x)-cd_0(y).$
     \end{itemize}
 \end{lemma}
 \begin{remark}
     In the case of standard homogeneous norm, the constant $c$ above is replaced by $1$. In that case we may assume that $\hn$ is a metric space equipped with the metric $|\cdot^{-1}\circ \cdot|$ (for details see \cite[Example 5.1]{BFS18_27}). Moreover, this metric is symmetric i.e. $|x^{-1}\circ y|=|y^{-1}\circ x|$ for any $x,y\in\hn.$
 \end{remark}
 Also note that the {\em Haar measure} on $\hn$ is the same as the {\em Lebesgue measure} on $\mathbb{R}^{2N+1}$ and satisfies the following the {\em doubling measure} property, that is for any $R>0$ and $x_0\in\hn$, 
 $$|B_{2R}(x_0)|\leq C(Q)|B_R(x_0)|$$
 This properties make $(\hn,|\cdot^{-1}\circ \cdot|,\mathfrak{L}^{2N+1})$ into a $Q$-regular measure metric space with a doubling measure (for more, see \cite[Section 11.3]{HaK00_28} and \cite[Proposition 14.2.9]{HKST15_29}).
\subsection{Auxiliary lemmas}In this subsection, we recall some important lemmas which will be used in the sequel. 
\begin{lemma}\cite[Lemma 2.6]{MPPP23_3}\label{L1.1}
    Let $x_0\in\hn$ and $|\cdot|_{\hn}$ be the standard homogeneous norm  $\hn$. Then for some $\gamma,r>0$,
    \begin{align*}
        \int_{B_r(x_0)}\frac{1}{\dist[x_0]{x}^{Q+\gamma}}dx\leq \frac{C(Q,\gamma)}{r^\gamma}
    \end{align*}
\end{lemma}
Next lemma is useful for iteration. 
\begin{lemma}\label[Lemma]{L1.3}\cite[Chapter 1, Lemma 4.1]{DiB93_4}
    Let $M,b>1$ and $\kappa,\delta>0$ be given. For every $n\in\mathbb{N}$, we assume that
		\begin{align*}
			Y_{n+1}\leq Mb^n\big(Y_n^{1+\delta}+Z_n^{1+\kappa}Y_h^\delta\big)\,\mbox{ and }\, Z_{n+1}\leq Mb^n\big(Y_n+Z_n^{1+\kappa}\big) .
		\end{align*}
		Moreover, assume that $Y_0+Z_0^{1+\kappa}\leq (2M)^{-\frac{1+\kappa}{\zeta}}b^{-\frac{1+\kappa}{\zeta^2}}$ for $\zeta=\min\{\kappa,\delta\}$. Then we have the following
		\begin{align*}
			\lim_{n\rightarrow\infty}Y_n=\lim_{n\rightarrow\infty}Z_n=0 .\\
		\end{align*}
\end{lemma}
Next, we provide a particular form of  Parabolic Sobolev inequality which can be deduced from the proof of \cite[Theorem 1.10]{KT25_1} as well as through \cite[Prop.~A.3]{Lia24_2}. This type of form is useful during the proof of De Giorgi Lemma (Lemma \ref{L3.1}).
\begin{proposition}\label{P1.7} Let $s\in (0,1)$ and
    \begin{align*}
    \kappa_*:=\begin{cases}
        & \frac{Q}{Q-sp},\;\;\;\; sp<Q,\\
        & \;\;\;2, \;\;\;\;\;\;\; sp\geq Q.
    \end{cases}
    \end{align*}
    Let $v\in L^p(I;W^{s,p}(B_R(\xi_0)))\cap L^\infty(I;L^2(B_R(\xi_0)))$ be compactly supported in $B_{(1-d)R}(\xi_0)$, for some $d\in (0,1)$ and a.e. $t\in I$, it holds that
    \begin{align}
        &\left(\int_I\int_{B_R}f(x,t)^{\kappa p}\,dx\,dt\right)\nonumber\\
            &\qquad \leq C\left(R^{sp}\int_I[f(\cdot,t)]^p_{W^{s,p}(B_R(\xi_0))}+\frac{1}{d^{Q+sp}}\|f(\cdot,t)\|^p_{L^p(B_R(\xi_0))}\,dt\right)\left(\sup_{t\in I}\fint_{B_R(\xi_0)}|f(\xi,t)|^2\,d\xi\right)^{\frac{\kappa_*-1}{\kappa_*}}
    \end{align}
    for some constant $C=C(Q,p,s)$ and 
    \begin{align*}
    \kappa:=1+\frac{2(\kappa_*-1)}{p\kappa_*}.
    \end{align*}
\end{proposition}

\subsection{Plan of the Paper} The paper is divided in several sections. In Section \ref{S2}, we present the energy estimate followed by the Section \ref{S3}, where we discuss some important lemmas namely, De Giorgi lemma (Lemma \ref{L3.1}), De Giorgi in forward time (Lemma \ref{L3.2}), measure theoretical information  forward in time (Lemma \ref{L3.3}) and shrinking lemma (Lemma \ref{L3.4}). Section \ref{S4} starts with expansion of positivity result for $1<p\leq 2$ condition and concludes with the proof of the general modulus of continuity result in the singular case. Section \ref{S5} is devoted for the $p>2$ case. Section \ref{S6} is dedicated to the proof of main result \autoref{T1.1}.
\subsection{Acknowledgments}
The author is financially supported by CSIR under the file number: $09/0106$
$(13571)/2022$-EMR-I. The author is grateful to Vivek Tewary for suggesting this problem. The author also expresses his gratitude to Krea University for providing the necessary funding and hospitality during the author's visit at Krea University, where this work was conceived.
\section{Caccioppoli inequality} \label{S2}
We want to consider the energy estimate in \cite[Lemma 3.1]{KT25_1} of the truncated function at the time-dependent level. After remodification of the estimate, we can have the following estimate
\begin{lemma}\label{L2.1}
Let $(x_0,t_0)\in\Omega_T$ and $l,r,R\in \mathbb{R}^+$ satisfying $0<l<r<R$. Take ball $B_r\equiv B_r(x_0)$ such that $\bar{B_r}\subset\Omega$. Consider non-negative functions $\psi\in C_0^\infty(B_{(r+l)/2}(x_0))$ with $0\leq\psi\leq 1$ in $\mathbb{H}^N$ and $\nu\in C^\infty(\mathbb{R})$ and $\theta_2>0$ satisfying $(t_0-\theta_2,t_0)\subset(t_0-R^{sp},t_0)\subset(0,T)$. If $u$ is a local sub(super)solution to the problem (\ref{main}) then there exist a constant $c=c(Q,p,s,\lambda',\Lambda')$ such that
    \begin{align}\label{eq2.1}
    \int_{t_0-\theta_2}^{t_0}&\iint_{B_r\times B_r}\frac{|(w_\pm\psi\nu^{2/p})(x,t)-(w_\pm\psi\nu^{2/p})(y,t)|^p}{|y^{-1}\circ x|^{Q+sp}_{\mathbb{H}^N}}dxdydt\nonumber\\
    &+\int_{t_0-\theta_2}^{t_0}\int_{B_r}(w_\pm\psi^p\nu^2)(x,t)\Bigg\{\int_{B_r}\frac{w^{p-1}_\mp(y,t)}{|y^{-1}\circ x|^{Q+sp}_{\mathbb{H}^N}}dy\Bigg\}dxdt+\int_{B_r}(w^2_\pm\psi^p\nu^2)(x,t)dx\Bigg|_{t_0-\theta_2}^{t_0}\nonumber\\
    &\quad \leq C\int_{t_0-\theta_2}^{t_0}\iint_{B_r\times B_r}\max\{w_\pm^p(x,t),w^p_\pm(y,t)\}\frac{|\psi(x)-\psi(y)|^p}{|y^{-1}\circ x|^{Q+sp}_{\mathbb{H}^N}}\nu^2(t)\;dxdydt\nonumber\\
    &\qquad +\int_{t_0-\theta_2}^{t_0}\int_{B_r}(w_\pm\psi^p\nu^2)(x,t)\Bigg\{\int_{\mathbb{H}^N\setminus B_r}\frac{w^{p-1}_\pm(y,t)}{|y^{-1}\circ x|^{Q+sp}_{\mathbb{H}^N}}dy\Bigg\}dxdt\\
    &\quad\qquad\mp\mathfrak{C}\int_{t_0-\theta_2}^{t_0}\int_{B_r}\Bigg(\int_{\hn\setminus B_R}\frac{u_\pm^{p-1}(y,t)}{\dist[x]{y}^{Q+sp}}dy\Bigg)(w_\pm\psi^p\nu^2)(x,t)dxdt\nonumber\\
    &\qquad\qquad +\int_{t_0-\theta_2}^{t_0}\int_{B_r}(w_\pm\psi^p\nu^2)(x,t)\partial_t\nu(t)dxdt\nonumber
\end{align}
where $w:=u-g-k$ with level $k\in\mathbb{R}$ and $g(t)=\mathfrak{C}\int_{t_0-R^{sp}}^t\Big(\int_{\hn\setminus B_R}\frac{u_\pm^{p-1}(y,t)}{\dist[x]{y}^{Q+sp}}dy\Big)dt$ is arbitrary.
\end{lemma}
\begin{proof}
    The proof follows similarly to \cite[Lemma 3.1]{KT25_1} with some obvious modifications. 
\end{proof}

\section{Preliminary Results}\label{S3}
In this section, we will discuss some main tools which are important to prove the H\"older estimate. We mainly follow the ideas in \cite{Lia24_5}. In each lemma, tail appears in either-or form which signifies one needs to control the tail to have the required results. We derive all these lemmas by using the estimate (\ref{eq2.1}).

Throughout this section, we use the following notations,
\begin{align*}
    \begin{cases}
        &\mathcal{Q}:=B_R(x_0)\times(T_1,T_2]\subset\Omega_T,\\
        &\mu^+\geq \esssup_\mathcal{Q}u,\,\,\,\,\mu^-\leq\essinf_\mathcal{Q}u,\\
        & \omega\geq\mu^+-\mu^-.
    \end{cases}
\end{align*}

Moreover note that, throughout this paper we use $|B_{mr}|=\mathcal{C}|B_r|$, where $\mathcal{C}$ depends upon $Q,m(>0)$ and the homogeneous norm which we are using on Heisenberg group. \vskip 1mm
This section will start with the De Giorgi Lemma where we consider measure theoretical information to derive a pointwise estimate. 
\subsection{De Giorgi Lemma}\label{De Giorgi}
\begin{lemma}\label{L3.1}
     Set $\theta=\delta(\xi\omega)^{2-p}$ for some $\xi,\delta\in (0,1)$ and for $0<\rho<R/2$, assume $z_0+\mathcal{Q}^\ominus_\rho(\theta)\subset B_R(x_0)\times (T_1,T_2]$. Then for locally bounded, local weak sub(super)solution $u$ to (\ref{main}), there exist $\tilde{\gamma}=\tilde{\gamma}(\textup{\texttt{data}})(>1)$ and $\iota\in(0,1)$ depending on \textup{\texttt{data}} and $\delta$, such that if
     \begin{align*}
         |\{\pm(\mu^\pm-u)\leq \xi\omega\}\cap (z_0+\mathcal{Q}_\rho^\ominus(\theta))|\leq \iota\,|z_0+\mathcal{Q}_\rho^\ominus(\theta)|
     \end{align*}
     then either
     \begin{align*}
         \tilde{\gamma}\textup{Tail }[(u-\mu^\pm)_\pm;\mathcal{Q}]>\xi\omega
     \end{align*}
     or
     \begin{align*}
         \pm(\mu^\pm-u)\geq \frac{1}{4}\xi\omega\,\,\,\mbox{ a.e. in }z_0+\mathcal{Q}^\ominus_{\rho/2}(\theta).
     \end{align*}
\end{lemma}
\begin{proof}
    Assume that $(x_0,t_0)=(0,0)$. We will prove this result for the supersolution case. For $n=0,1,2,3,...$, we define $\rho_n,\tilde{\rho}_n,\hat{\rho}_n,\bar{\rho}_n,B_n,\tilde{B}_n,\hat{B}_n,\bar{B}_n,\mathcal{Q}^\ominus_n,\tilde{\mathcal{Q}}^\ominus_n,\hat{\mathcal{Q}}^\ominus_n,\bar{\mathcal{Q}}^\ominus_n $ as in \cite[Lemma 3.1]{Lia24_2} viz.,
    \begin{align*}
        \begin{cases}
            &\rho_n=\frac{\rho}{2}+\frac{\rho}{2^{n+1}},\,\,\tilde{\rho}_n=\frac{\rho_n+\rho_{n+1}}{2},\\
            &\hat{\rho}_n=\frac{\rho_n+\tilde{\rho}_n}{2},\,\,\bar{\rho}_n=\frac{\tilde{\rho}_n+\rho_{n+1}}{2}\\
            & B_n=B_{\rho_n}(0),\,\tilde{B}_n=B_{\tilde{\rho}_n}(0),\,\hat{B}_n=B_{\hat{\rho}_n}(0),\,\bar{B}_n=B_{\bar{\rho}_n}(0),\\
            &\mathcal{Q}_n^\ominus=B_n\times (-\theta\rho_n^{sp},0],\,\,\tilde{\mathcal{Q}}_n^\ominus=\tilde{B}_n\times (-\theta\tilde{\rho}_n^{sp},0],\\
            & \hat{\mathcal{Q}}_n^\ominus=B_n\times (-\theta\hat{\rho}_n^{sp},0],\,\,\bar{\mathcal{Q}}_n^\ominus=\bar{B}_n\times (-\theta\tilde{\rho}_n^{sp},0].
        \end{cases}
    \end{align*}
    It is obvious that
    \begin{align*}
        \mathcal{Q}^\ominus_{n+1}\subset\bar{\mathcal{Q}}_n^\ominus\subset\tilde{\mathcal{Q}}_n^\ominus\subset\hat{\mathcal{Q}}_n^\ominus\subset\mathcal{Q}^\ominus_n.
    \end{align*}
    Use (\ref{eq2.1}) with domain $B_n$ and $\mathcal{Q}^\ominus_n$ to $(u-g-k_n)_+$, with the level $g(t)+k_n$ where
    \begin{align*}
        k_n:=\mu^+-\xi_n\omega,\;\;\;\;\;\mbox{ where } \xi_n=\frac{\xi}{2}+\frac{\xi}{2^{n+1}}
    \end{align*}
    Consider the following with an arbitrary positive constant $\mathfrak{C}$ (to be chosen later),
    \begin{align*}
        w_+(x,t):=(u-g-k_n)_+,\;\;\;\;\;\;\;g(t)=\mathfrak{C}\int_{{-R}^{-sp}}^t\int_{\mathbb{H}^N\setminus B_R}\frac{|u_+(y,\tau)|^{p-1}}{|y^{-1}|_{\hn}^{Q+sp}}dyd\tau
    \end{align*}
    Define two cut-off functions $\psi$ and $\nu$ in $B_n$ and $(-\theta\rho_n^{sp},0)$ respectively,
    \begin{align*}
        \psi=\begin{cases}
            & 1\;\;\;\mbox{in } \tilde{B}_n\\
            & 0\;\;\;\mbox{in } \hat{B}_n^c
        \end{cases}
       \;\;\;\;\;\;\;\;\;\;\;\;\; \mbox{ such that } |\nabla_{\hn}\psi|\leq \gamma\frac{2^n}{\rho}
    \end{align*}
    and
    \begin{align*}
        \nu=\begin{cases}
            & 1\;\;\;\mbox{in } (-\theta\tilde{\rho}^{sp},0)\\
            & 0\;\;\;\mbox{in } (-\theta\hat{\rho}_n^{sp},0)^c
        \end{cases}
        \;\;\;\; \mbox{ such that }|\partial_t\nu|\leq \gamma\frac{2^{spn}}{\theta\rho^{sp}}
    \end{align*}
    Considering all these entities, (\ref{eq2.1}) implies (removing the {\em good term} as it is non-negative)
    \begin{align}\label{eq3.1}
        \esssup_{-\theta\tilde{\rho}_n^{sp}<t<0}&\intb[\tilde{B}]{n}w^2_+(x,t)dx+\int_{-\theta\tilde{\rho}_n^{sp}}^0\intb[\tilde{B}]{n}\intb[\tilde{B}]{n}\gaglifract[w_+]dxdydt\nonumber\\
        &\leq \gamma \int_{-\theta\rho_n^{sp}}^0\intb[B]{n}\intb[B]{n}\max\{w^p_+(x,t),w_+^p(y,t)\}\gaglifrac[\psi]\nu^2(t)dxdydt\nonumber\\
        &\quad+\gamma\int_{-\theta\rho_n^{sp}}^0\intb[B]{n}(w_+\psi^p\nu^2)(x,t)\Bigg\{\tail[w_+]{x}{y}{n}{B}dy\Bigg\}dxdt\\
        &\qquad-\mathfrak{C}\int_{-\theta\rho_n^{sp}}^0\intb[B]{n}g'(t)(w_+\psi^p\nu^2)(x,t)dxdt+\int_{-\theta\rho_n^{sp}}^0\intb[B]{n}(w_+\psi^p\nu^2)(x,t)\partial_t\nu(t)dxdt\nonumber
    \end{align}
We make the following estimation of the terms on the right-hand side of (\ref{eq3.1}) successively. In the case of the first term,
\begin{align}\label{eq3.2}
    \int_{-\theta\rho_n^{sp}}^0\intb[B]{n}\intb[B]{n}&\max\{w^p_+(x,t),w_+^p(y,t)\}\gaglifrac[\psi]\nu^2(t)dxdydt\nonumber\\
    &\leq \frac{2^{np}}{\rho^p}.2(\xi\omega)^p\int_{-\theta\rho_n^{sp}}^0\intb[B]{n}\intb[B]{n}\frac{\chi_{\{u(x,t)-g(t)>k_n\}}}{|y^{-1}\circ x|^{Q+(s-1)p}_{\mathbb{H}^N}}dxdydt\nonumber\\
    & = \frac{2^{np}}{\rho^p}.2(\xi\omega)^p\int_{-\theta\rho_n^{sp}}^0\intb[B]{n}\chi_{\{u(x,t)-g(t)>k_n\}}dx\Bigg[\intb[B]{n}\frac{1}{|y^{-1}\circ x|^{Q+(s-1)p}_{\mathbb{H}^N}}dy\Bigg]dt\nonumber\\
    &\leq \gamma\frac{2^{np}}{\rho^{sp}}(\xi\omega)^p|\mathcal{A}_n|
\end{align}
In the last line, we use the fact that $\mathcal{A}_n:=\{(x,t)\in \mathcal{Q}^\ominus_n: u(x,t)-g(t)> k_n\}$. \\
Prior to the estimations of the second and third terms, we make the following observations,\\
$\bullet$ If $|y|>\rho_n$ and $|x|<\hat{\rho}_n$, then,
\begin{align*}
    \frac{|y^{-1}\circ x|}{|y|}\geq 1-\frac{\hat{\rho}_n}{\rho_n}&=\frac{1}{4}\Big(\frac{\rho_n-\rho_{n+1}}{\rho_n}\Big)\\
    &\geq \frac{1}{2^{n+4}}
\end{align*}
$\bullet$ When $|y|>R$ and $|x|\leq \rho$, then
\begin{align*}
    \frac{|y^{-1}\circ x|}{|y|}\geq \frac{1}{2},\;\;\;\;\;\;\mbox{provided }\rho\leq R/2
\end{align*}
Considering all the facts along with the fact that $u\geq 0$ a.e.  in $\mathcal{Q}$, we have in the second term as follows
\begin{align}\label{eq3.3}
    \int_{-\theta\rho_n^sp}^0&\intb[B]{n}(w_+\psi^p\nu^2)(x,t)\Bigg\{\tail[w_+]{x}{y}{n}{B}dy\Bigg\}\nonumber\\&= \int_{-\theta\rho_n^sp}^0\intb[B]{n}(w_+\psi^p\nu^2)(x,t)\Bigg\{\tail[w_+]{x}{y}{R}{B}dy+\int_{B_R\setminus B_n}\frac{w_+^{p-1}}{{\dist[x]{y}}^{Q+sp}}\Bigg\}\nonumber\\
    &\leq \gamma 2^{(Q+sp)n} \int_{-\theta\rho^{sp}}^0\intb[B]{n}(w_+\psi^p\nu^2)(x,t)dx\Bigg\{\int_{B_R\setminus B_n}\frac{w_+^{p-1}}{|y|_{\mathbb{H}^N}^{Q+sp}}dy\Bigg\}dt \nonumber\\
    &\quad +\gamma \int_{-\theta\rho^{sp}}^0\intb[B]{n}(w_+\psi^p\nu^2)(x,t)dx\Bigg\{\int_{\mathbb{H}^N\setminus B_R}\frac{w_+^{p-1}}{|y|_{\mathbb{H}^N}^{Q+sp}}dy\Bigg\}dt\nonumber\\
    &\leq \gamma 2^{(Q+sp)n}\frac{(\xi\omega)^{p-1}}{\rho^{sp}}\int_{-\theta\rho^{sp}}^0\intb[B]{n}(w_+\psi^p\nu^2)(x,t)dxdt\nonumber\\
    &\quad+\gamma\int_{-\theta\rho^{sp}}^0\intb[B]{n}(w_+\psi^p\nu^2)(x,t)dxdt\Bigg\{\int_{\mathbb{H}^N\setminus B_R}\frac{u_+^{p-1}}{|y|_{\mathbb{H}^N}^{Q+sp}}dy\Bigg\}
\end{align}
in this series of calculations, we use the Lemma \ref{L1.1} in first integral of last step. Changing the arbitrary constant $\gamma$ to $\mathfrak{C}$ in the last integral of the last step, then it will cancel out with the third term of (\ref{eq3.1}. Then the rest of (\ref{eq3.3}) will be dominated by 
\begin{align}\label{eq3.4}
    \gamma 2^{(Q+sp)n}\frac{(\xi \omega)^p}{\rho^{sp}}|\mathcal{A}_n|
\end{align}
The last term of (\ref{eq3.1}) is standard, namely
\begin{align}\label{eq3.5}
    \int_{-\theta\rho_n^{sp}}^0\intb[B]{n}(w_+\psi^p\nu^2)(x,t)\partial_t\nu(t)dxdt\leq \frac{2^{spn}}{\theta\rho^{sp}}(\xi\omega)^2|\mathcal{A}_n|
\end{align}
Hence, the energy estimate (\ref{eq3.1}) can be rewritten with the help of the estimations (\ref{eq3.2}), (\ref{eq3.4}) and (\ref{eq3.5})
\begin{align*}
    \esssup_{-\theta\tilde{\rho}_n^{sp}<t<0}&\intb[\tilde{B}]{n}w^2_+(x,t)dx+\int_{-\theta\tilde{\rho}_n^{sp}}^0\intb[\tilde{B}]{n}\intb[\tilde{B}]{n}\gaglifract[w_+]dxdydt\leq \gamma 2^{(Q+sp)n}\frac{\xi\omega}{\delta\rho^{sp}}|\mathcal{A}_n|
\end{align*}
Now set $0\leq \psi\leq 1$  and $0\leq\nu\leq 1$ to be a cut-off function on $\tilde{B}_n$ and $(-\theta\tilde{\rho}_n^{sp},0)$ such that
\begin{align*}
    \psi=\begin{cases}
        & 1\;\;\;\;\mbox{ on } B_{n+1}\\
        & 0\;\;\;\;\mbox{ on } \bar{B}_n^c
    \end{cases}
    \;\;\;\;\; \mbox{such that } |\nabla_{\hn}\psi|\leq \gamma\frac{2^n}{\rho}.
\end{align*}
and 
\begin{align*}
    \nu=\begin{cases}
        & 1\;\;\;\;\mbox{ on } (-\theta\rho_{n+1}^{sp},0)\\
        & 0\;\;\;\;\mbox{ on } (-\theta\bar{\rho}_n^{sp},0)^c
    \end{cases}
    \;\;\;\;\; \mbox{ such that } |\nu_t|\leq\gamma\frac{2^{spn}}{\theta\rho^{sp}}.
\end{align*}
Consequently, using H\"older inequality and Sobolev embedding, we have the following
\begin{align*}
    \frac{\xi\omega}{2^{n+2}}|\mathcal{A}_{n+1}|&\leq \int_{-\theta\tilde{\rho}_n^{sp}}^0\int_{\tilde{B}_n}(w_+\psi)(x,t)dxdt\\
    &\leq \Bigg[\int_{-\theta\tilde{\rho}_n^{sp}}^0\intb[\tilde{B}]{n}(w_+\psi)^{\kappa p}dxdt\Bigg]^\frac{1}{\kappa p}|\mathcal{A}_n|^{1-\frac{1}{\kappa p}}\\
    & \leq \gamma\Bigg[\rho^{sp}\int_{-\theta\tilde{\rho}_n^{sp}}^0\intb[\tilde{B}]{n}\intb[\tilde{B}]{n}\gaglifract[w_+\psi]dxdydt+2^{(Q+sp)n}\int_{-\theta\tilde{\rho}_n^{sp}}^0\intb[\tilde{B}]{n}(w_+\psi)^pdxdt\Bigg]^\frac{1}{\kappa p}\\
    & \quad \times \Bigg[\esssup_{-\theta\tilde{\rho}_n^{sp}<t<0}\fint(w_+\psi)^2dx\Bigg]^\frac{\kappa_*-1}{\kappa_*\kappa p}|\mathcal{A}_n|^{1-\frac{1}{\kappa p}}\\
    & \leq \gamma\Bigg[\rho^{sp}\int_{-\theta\tilde{\rho}_n^{sp}}^0\intb[\tilde{B}]{n}\intb[\tilde{B}]{n}\gaglifract[w_+]dxdydt+2^{np}\int_{-\theta\tilde{\rho}_n^{sp}}^0\intb[\tilde{B}]{n}w^p_+(x,t)dxdt\\
    &\quad + 2^{(Q+sp)n}\int_{-\theta\tilde{\rho}_n^{sp}}^0\intb[\tilde{B}]{n}(w_+\psi)^pdxdt\Bigg]^\frac{1}{\kappa p}\times \Bigg[\esssup_{-\theta\tilde{\rho}_n^{sp}<t<0}\fint(w_+\psi)^2dx\Bigg]^\frac{\kappa_*-1}{\kappa_*\kappa p}|\mathcal{A}_n|^{1-\frac{1}{\kappa p}}\\
    &\leq \gamma b^n\delta^{-\big[\frac{1}{\kappa}+\frac{2\kappa_*-1}{\kappa_*\kappa p}\big]}\rho^{-(Q+sp)\frac{\kappa_*-1}{\kappa_*\kappa p}}(\xi\omega)^\frac{2\kappa_*-1}{\kappa_*\kappa}|\mathcal{A}_n|^{1+\frac{\kappa_*-1}{\kappa_*\kappa p}}
\end{align*}
for some $b=b(Q,p)$. In this sequence of calculations, the third inequality comes from Proposition \ref{P1.7} with 
\begin{align*}
    \kappa_*:=\begin{cases}
        & \frac{Q}{Q-sp},\;\;\;\; sp<Q\\
        & \;\;\;2, \;\;\;\;\;\;\; sp\geq Q
    \end{cases}
\end{align*}
and
\begin{align*}
    \kappa:=1+\frac{2(\kappa_*-1)}{p\kappa_*},\;\;d=\frac{1}{2^{n+4}}.
\end{align*}
To obtain the second-to-last line, we use result of the following calculation viz,
\begin{align*}
    &\rho^{sp}\int_{-\theta\tilde{\rho}_n^{sp}}^0\intb[\tilde{B}]{n}\intb[\tilde{B}]{n}\gaglifract[w_+\psi]dxdydt\\
    &\qquad \leq C\rho^{sp}\int_{-\theta\tilde{\rho}_n^{sp}}^0\intb[\tilde{B}]{n}\intb[\tilde{B}]{n}\gaglifract[w_+]dxdydt\\
    &\qquad\quad + C\rho^{sp}\int_{-\theta\tilde{\rho}_n^{sp}}^0\intb[\tilde{B}]{n}\intb[\tilde{B}]{n}w^p_+(y,t)\frac{|\psi(x)-\psi(y)|^p}{\dist[x]{y}^{Q+sp}}dxdydt\\
    &\qquad \leq C\rho^{sp}\int_{-\theta\tilde{\rho}_n^{sp}}^0\intb[\tilde{B}]{n}\intb[\tilde{B}]{n}\gaglifract[w_+]dxdydt\\
    &\qquad\quad+\gamma 2^{np}\iint_{\tilde{\mathcal{Q}}^\ominus_n}w_+^p(y,t)dydt,
\end{align*}
 where $C=C(p)$ and last inequality is obtained by employing energy estimate (\ref{eq3.5}). Now this estimate leads to a recursive inequality by $Y_n=\frac{|\mathcal{A}_n|}{|Q_n|}$, viz., there is some constant $\gamma$ such that the following holds:
\begin{align*}
    Y_{n+1}\leq \gamma (2b)^n\delta^{-\big[\frac{1}{\kappa}+\frac{2\kappa_*-1}{\kappa_*\kappa p}\big]}Y_n^{1+\frac{\kappa_*-1}{\kappa_*\kappa p}}.
\end{align*}
Hence, by the convergence result Lemma \ref{L1.3}, we have a constant $\iota=\iota(\textup{\texttt{data}})$ such that if
\begin{align*}
    |\{u-g(t)\geq -\xi\omega+\mu^+\}\cap \mathcal{Q}^\ominus_\rho(\theta)|\leq \iota|\mathcal{Q}^\ominus_\rho(\theta)|
\end{align*}
then
\begin{align*}
    u(x,t)-g(t)\leq -\frac{1}{2}\xi\omega + \mu^+\;\;\;\;\;\mbox{ a.e. in } \mathcal{Q}^\ominus_{\frac{\rho}{2}}(\theta) 
\end{align*}
Consequently, if $\sup _t g(t)\leq \frac{1}{4}\xi\omega$, redefining $\mathfrak{4C}$ as $\mathfrak{C}$ and if
\begin{align*}
    |\{u\geq -\xi\omega+\mu^+\}\cap \mathcal{Q}^\ominus_\rho(\theta)|\leq \iota|\mathcal{Q}^\ominus_\rho(\theta)|
\end{align*}
then
\begin{align*}
    u\leq -\frac{1}{4}\xi\omega+\mu^+\;\;\;\;\;\mbox{ a.e. in } \mathcal{Q}^\ominus_{\frac{\rho}{2}}(\theta)
\end{align*}
\end{proof}
Next, we prove a forward-in-time version of De Giorgi lemma but considering pointwise initial data unlike lemma \ref{L3.1} where we have used measure-theoretical data. 
\subsection{De Giorgi Lemma : Forward in Time}
\begin{lemma}\label{L3.2}
    Let $\xi\in (0,1)$ and $u$ be a locally bounded weak sub(super)solution of the equation (\ref{main}) in $\Omega_T$. There exist $\tilde{\gamma}(>1)$ and $\iota_0\in(0,1)$, depending only on the \textup{\texttt{data}} and independent of $\xi$, such that $B_{\rho/2}(x_0)\times(t_0,t_0+\iota_0(\xi\omega)^{2-p}\rho^{sp}]\subset\mathcal{Q}$ and if 
    $$\pm(\mu^\pm-u(\cdot,t_0))\geq \xi\omega\;\;\;\;\;\mbox{ a.e. in } B_\rho(x_0),$$
    then either $$\tilde{\gamma}\;\textup{Tail }[(u-\mu^\pm)_\pm;\mathcal{Q}]>\xi\omega$$
    or
    $$\pm(\mu^\pm-u)\geq \frac{1}{4}\xi\omega\;\;\;\;\;\mbox{ a.e. in }B_{\rho/2}(x_0)\times(t_0,t_0+\iota_0(\xi\omega)^{2-p}\rho^{sp}].$$
\end{lemma}
\begin{proof}
    We show this result for the subsolution by assuming $(x_0,t_0)=(0,0)$. The supersolution counterpart of this result follows in a similar procedure. Comparing with the previous lemma, we make two changes here, namely
    \begin{itemize}
        \item We prove this lemma in the forward in time domain, i.e.  $\mathcal{Q}^\oplus_\rho(\theta)$ type domains and examine the energy estimate (\ref{eq2.1}) on that domain, where $\theta$ will be determined in the sequel.
        \item As far as cut-off functions are concerned, we only consider {\em time-independent} cut-off functions, i.e. $\partial\nu_t=0$. 
    \end{itemize}
    Consider the functions $w_+$ and $g(t)$ as in the previous lemma \ref{L3.1}. If we take $k\geq -\xi\omega+\mu^+$, then examining energy estimate (\ref{eq2.1}) for $(u-g-k)_+$ in the cylinder $\mathcal{Q}^\oplus_\rho(\theta)$ at time level $t=0$ vanishes in view of the term at the time level $t_0-\theta_2$ in Lemma \ref{L2.1}. Let us introduce $\rho_n,\tilde{\rho}_n,\hat{\rho}_n,\bar{\rho}_n,B_n,\tilde{B}_n,\hat{B}_n,\bar{B}_n,$ and $k_n,\xi_n$ as in \cite[Lemma 3.1]{Lia24_2} and the previous lemma \ref{L3.1} respectively, viz.
    \begin{align*}
        \begin{cases}
            &\rho_n=\frac{\rho}{2}+\frac{\rho}{2^{n+1}},\,\,\tilde{\rho}_n=\frac{\rho_n+\rho_{n+1}}{2},\\
            &\hat{\rho}_n=\frac{\rho_n+\tilde{\rho}_n}{2},\,\,\bar{\rho}_n=\frac{\tilde{\rho}_n+\rho_{n+1}}{2}\\
            & B_n=B_{\rho_n}(0),\,\tilde{B}_n=B_{\tilde{\rho}_n}(0),\\
            &\hat{B}_n=B_{\hat{\rho}_n}(0),\,\bar{B}_n=B_{\bar{\rho}_n}(0),\\
            & k_n= \mu^+-\xi_n\omega,\,\,\xi_n=\frac{\xi}{2}+\frac{\xi}{2^{n+1}},
        \end{cases}
    \end{align*}
    and parabolic {\em forward} cylinders as follows
    \begin{align*}
        \begin{cases}
            &B_n\times (0,\theta\rho^{sp}],\tilde{B}_n\times(0,\theta\rho^{sp}],\\
            &\hat{B}_n\times(0,\theta\rho^{sp}],\bar{B}_n\times(0,\theta\rho^{sp}].
        \end{cases}
    \end{align*}
     Here, note that balls are shrinking along $\rho_n$ while the height of the cylinders remains fixed, namely $\theta\rho^{sp}$. The cut-off function $\psi(x)$ in $B_n$ is chosen in the following way
    \begin{align*}
        \psi\equiv\begin{cases}
            & 1\;\;\;\;\mbox{ on } \tilde{B}_n,\\
            & 0\;\;\;\;\mbox{ on } \hat{B}_n^c,
        \end{cases}
        \;\;\;\;\;\;\;\mbox{ such that } |\nabla_{\hn}\psi|\leq \gamma\frac{2^{n+4}}{\rho}.
    \end{align*}
With these choices, the energy estimate (\ref{eq2.1}) becomes
\begin{align}\label{eq3.6}
    \esssup_{0<t<\theta\rho^{sp}}&\intb[\tilde{B}]{n}w^2_+(x,t)dx+\int_0^{\theta\rho^{sp}}\intb[\tilde{B}]{n}\intb[\tilde{B}]{n}\frac{|w_+(x,t)-w_+(y,t)|}{\dist[x]{y}^{Q+sp}}dxdydt\nonumber\\
    &\leq \gamma \int_0^{\theta\rho^{sp}}\intb[\tilde{B}]{n}\intb[\tilde{B}]{n}\max\{w^p_+(x,t),w^p_+(y,t)\}\gaglifrac[\psi]dxdydt\nonumber\\
    &\quad +\gamma \int_0^{\theta\rho^{sp}}\intb[B]{n}(w_+\psi^p)(x,t)\Bigg\{\tail[w_+]{x}{y}{n}{B}dy\Bigg\}dxdt\\
    &\qquad -\mathfrak{C}\int_0^{\theta\rho^{sp}}\intb[B]{n}R^{-sp}\mbox{ Tail}^{p-1}(u_+(t);R,x_0)(w_+\psi^p)(x,t)dxdt\nonumber
\end{align}
Consequently, we have 
\begin{align*}
    \esssup_{0<t<\theta\rho^{sp}}\intb[\tilde{B}]{n}w^2_+(x,t)dx&+\int_0^{\theta\rho^{sp}}\intb[\tilde{B}]{n}\intb[\tilde{B}]{n}\frac{|w_+(x,t)-w_+(y,t)|}{\dist[x]{y}^{Q+sp}}dxdydt\\
    &\leq \gamma 2^{(Q+sp)n}\frac{(\xi\omega)^p}{\rho^{sp}}|\mathcal{A}_n|
\end{align*}
The estimate on the right-hand side of (\ref{eq3.6}) arises by using the same treatment as in the Lemma \ref{L3.1} with suitable reconstruction of the constant $\mathfrak{C}$ and $\gamma$. Moreover, we use the definition $\mathcal{A}_n:=\{(x,t)\in \mathcal{Q}_n:u(x,t)-g(t)>k_n\}$. After running the De Giorgi iteration on this last estimate, one can obtain a $\iota_0$, depending only on data, such that with the choice of $\theta=\iota_0(\xi\omega)^{2-p}$ we have
\begin{align*}
    u(x,t)-g(t)\leq \mu_+-\frac{1}{2}\xi\omega\;\;\;\;\;\mbox{ a.e. in }B_\frac{\rho}{2}\times (0,\iota_0(\xi\omega)^{2-p}\rho^{sp}].
\end{align*}
Consequently, if we impose $\sup_t g(t)\leq \frac{\xi\omega}{4}$ with redefining $\mathfrak{C}=4\mathfrak{C}$, then
\begin{align*}
    u\leq -\frac{\xi\omega}{4}+\mu_+\;\;\;\;\;\mbox{ a.e. in } B_\frac{\rho}{2}\times (0,\iota_0(\xi\omega)^{2-p}\rho^{sp}].
\end{align*}
\end{proof}
The following lemma propagates the measure theoretical results forward in time. 
\subsection{Measure Theoretical Information Forward in Time}
\begin{lemma}\label{L3.3}
    Let $\xi, \alpha\in(0,1)$. For locally bounded, local weak sub(super)solution $u$ of (\ref{main}) in $\Omega_T$ there exists $\delta,\varepsilon\in (0,1)$ depending only on the \textup{\texttt{data}} and $\alpha$ such that if
    $$|\{\pm(\mu^\pm-u(\cdot,t_0))\geq \xi\omega\}\cap B_\rho(x_0)|\geq \alpha|B_\rho(x_0)|,$$
    then either
    $$\frac{1}{\delta}\textup{Tail }[(u-\mu^\pm)_\pm;\mathcal{Q}]>\xi\omega$$
    or
    $$|\{\pm(\mu_\pm-u(\cdot,t)\geq \varepsilon\xi\omega\}\cap B_\rho(x_0)|\geq\frac{\alpha}{2}|B_\rho(x_0)|,\;\;\;\;\;\textup{ for all } t\in (t_0,t_0+\delta(\xi\omega)^{2-p}\rho^{sp}] $$
    provided $B_\rho(x_0)\times(t_0,t_0+\delta(\xi\omega)^{2-p}\rho^{sp}]\subset\mathcal{Q}$. Moreover, we have $\delta\approx\alpha^{p+Q+1}$ and $\varepsilon\approx\alpha$.
\end{lemma}
\begin{proof}
    Without loss of any generality, we assume that $(x_0,t_0)=(0,0)$. We will prove this lemma for subsolutions. Let us denote $aM:=-a\xi\omega+\mu^+$ for any real number $a$, and for any $t>0$
    \begin{align*}
        \mathcal{A}_{aM,\rho}(t):=\{u(\cdot,t)>aM\}\cap B_\rho
    \end{align*}
then clearly 
\begin{align}\label{eq3.7}
    |\mathcal{A}_{M,\rho}(0)|\leq (1-\alpha)|B_\rho|
\end{align}
To proceed further, we consider a time-independent cut-off function $\psi$ on $B_\rho$ as follows
\begin{align*}
    \psi\equiv\begin{cases}
        & 1 \;\;\;\;\mbox{ on } B_{(1-\sigma)\rho}\\
        & 0 \;\;\;\;\mbox{ on }B^c_{(1-\frac{\sigma}{2})\rho}
    \end{cases}
    \;\;\;\;\; \mbox{ such that } |\nabla_{\hn}\psi|\leq \frac{\gamma}{\sigma\rho}
\end{align*}
Examine the energy estimate (\ref{eq2.1}) over the domain $B_\rho\times(0,\delta(\xi\omega)^{2-p}\rho^{sp}]$ with the truncated function $w_+=(u-M)_+$. Then we have for all $t\in (0,\delta(\xi\omega)^{2-p}\rho^{sp}]$,
\begin{align}\label{eq3.8}
   \intb[B]{(1-\sigma)\rho}w^2_+(x,t)dx&\leq \intb[B]{\rho}w^2_+(x,0)dx+\gamma\int_0^{\delta(\xi\omega)^{2-p}\rho^{sp}}\intb[B]{\rho}\intb[B]{\rho}\max\{w^p_+(x,t),w^p_+(y,t)\}\gaglifrac[\psi]dxdydt\nonumber\\
   & \quad+\gamma\int_0^{\delta(\xi\omega)^{2-p}\rho^{sp}}\intb[B]{\rho}(w_+\psi^p)(x,t)\Bigg\{\tail[w_+]{x}{y}{\rho}{B}dy\Bigg\}dxdt
\end{align}
Estimate the right-hand side of (\ref{eq3.7}) one by one. Consider the second term
\begin{align*}
   \int_0^{\delta(\xi\omega)^{2-p}\rho^{sp}}&\intb[B]{\rho}\intb[B]{\rho}\max\{w^p_+(x,t),w^p_+(y,t)\}\gaglifrac[\psi]dxdydt \\
   &\leq \gamma \delta(\xi\omega)^{2-p}\rho^{sp}\frac{(\xi\omega)^{p}}{(\sigma\rho)^p}\intb[B]{\rho}\intb[B]{\rho}\frac{1}{\dist[x]{y}^{Q+sp-p}}dxdy\\
   &=\gamma\delta\frac{(\xi\omega)^2}{\sigma^p}|B_\rho|
\end{align*}
We observe that for $y\in\mathbb{H}^N\setminus B_\rho$ and $x\in \mbox{supp } \psi\subset B_{(1-\frac{\sigma}{2})\rho}$, we have
\begin{align*}
    \frac{\dist[x]{y}}{|y|_{\hn}}\geq 1-\frac{|x|_{\hn}}{|y|_{\hn}}\geq \frac{\sigma}{2}
\end{align*}
Using this observation, the third term has the following estimation
\begin{align*}
    \int_0^{\delta(\xi\omega)^{2-p}\rho^{sp}}&\intb[B]{\rho}(w_+\psi^p)(x,t)\Bigg\{\tail[w_+]{x}{y}{\rho}{B}dy\Bigg\}dxdt\\
    &\leq 2^{Q+sp}\frac{\xi\omega|B_\rho|}{\sigma^{Q+sp}}\Bigg\{\int_0^{\delta(\xi\omega)^{2-p}\rho^{sp}}\int_{B_R\setminus B_\rho}\frac{w_+^{p-1}(y,t)}{|y|_{\hn}^{Q+sp}}dydt+\int_0^{\delta(\xi\omega)^{2-p}\rho^{sp}}\int_{{\hn}\setminus B_R}\frac{w_+^{p-1}(y,t)}{|y|_{\hn}^{Q+sp}}dydt\Bigg\}\\
    &\leq \gamma\frac{\xi\omega|B_\rho|}{\sigma^{Q+sp}}\Bigg\{\gamma \delta\xi\omega+\int_0^{\delta(\xi\omega)^{2-p}\rho^{sp}}\int_{{\hn}\setminus B_R}\frac{w_+^{p-1}(y,t)}{|y|_{\hn}^{Q+sp}}dydt\Bigg\}\\
    &\leq \gamma\frac{\delta(\xi\omega)^2}{\sigma^{Q+sp}}|B_\rho|
\end{align*}
In this estimation, second last deduces from lemma \ref{L1.1} and  last line comes from the enforcement of 
\begin{align*}
    \frac{1}{\delta}\textup{Tail }[u_\pm;\mathcal{Q}]\leq \delta\xi\omega
\end{align*}
Hence, (\ref{eq3.8}) will be after these estimates,
\begin{align*}
     \intb[B]{(1-\sigma)\rho}w^2_+(x,t)dx&\leq \intb[B]{\rho}w^2_+(x,0)dx+\gamma\frac{\delta(\xi\omega)^2}{\sigma^{Q+sp}}|B_\rho|\\
     &\leq \Big[(1-\alpha)+\frac{\gamma\delta}{\sigma^{Q+p}}\Big](\xi\omega)^2|B_\rho|,
\end{align*}
for all $t\in (0,\delta(\xi\omega)^{2-p}\rho^{sp}]$,
where in the last step, we used (\ref{eq3.7}). The left-hand side has the following estimation for $\varepsilon\in(0,1)$,
\begin{align*}
    \intb[B]{(1-\sigma)\rho}w^2_+(x,t)dx&\geq \int_{B_{(1-\sigma)\rho}\cap\{u(\cdot,t)>\varepsilon M\}}w^2_+(x,t)dx\\
    &\geq (1-\varepsilon)^2(\xi\omega)^2|\mathcal{A}_{\varepsilon M,(1-\sigma)\rho}(t)|
\end{align*}
Also 
\begin{align*}
    |\mathcal{A}_{\varepsilon M,\rho}(t)|&=|\mathcal{A}_{\varepsilon M,(1-\sigma)\rho}(t)\cup(\mathcal{A}_{\varepsilon M,\rho}(t)-\mathcal{A}_{\varepsilon M,(1-\sigma)\rho}(t)|\\
    &\leq |\mathcal{A}_{\varepsilon M,(1-\sigma)\rho}(t)|+|B_\rho-B_{(1-\sigma)\rho}|\\
    &\leq|\mathcal{A}_{\varepsilon M,(1-\sigma)\rho}(t)|+Q\sigma|B_\rho|
\end{align*}
Combining the last three estimates, we have for all $t\in (0,\delta(\xi\omega)^{2-p}\rho^{sp}]$
\begin{align*}
    |\mathcal{A}_{\varepsilon M,\rho}(t)|\leq \frac{1}{(1-\varepsilon)^2}\Big[(1-\alpha)+\frac{\gamma\delta}{\sigma^{Q+p}}+Q\sigma\Big]|B_\rho|
\end{align*}
Consequently, we obtain for all $0\leq t\leq \delta(\xi\omega)^{2-p}\rho^{sp}$
\begin{align*}
    |\mathcal{A}_{\varepsilon M,\rho}(t)|\leq (1-\alpha/2)|B_\rho|
\end{align*}
after the choices of the parameters in the following way, i.e. 
\begin{align*}
    \sigma=\frac{\alpha}{8Q},\;\;\delta=\frac{1}{8\gamma}\sigma^{Q+p}\alpha\;\mbox{ and } (1-\varepsilon)\geq \sqrt{\frac{1-\frac{3}{4}\alpha}{1-\frac{1}{2}\alpha}}
\end{align*}
\end{proof}
Unlike the local case, there is no De Giorgi isoperimetric type inequality  result in non-local case due to the possible presence of {\em jump} in the functions of $W^{s,p}$ space. In his paper, Cozzi\cite{Coz17_6} proves De Giorgi isoperimetric inequality for operator whose prototype is the fractional p-Laplacian which is stable under $s\rightarrow 1$. But whenever $s$ moving away from $1$, he relies on the {\em good term} which is also present in our Caccioppoli estimate (\ref{eq2.1}). We will also use this {\em good term} to derive the following lemma.
\subsection{Shrinking Lemma}
\begin{lemma}\label{L3.4}
    Let $u$ be a locally bounded, local weak sub(super)solution to (\ref{main}) in $\Omega_T$. For some $\delta,\alpha$ and $\xi$ in $(0,1/2)$, set $\theta=\delta(\sigma\xi\omega)^{2-p}$ there exists $\gamma>1$ depending only on \textup{\texttt{data}} and independent of $\{\alpha,\delta,\sigma,\xi\}$ such that if 
    $$|\{\pm(\mu^\pm-u(\cdot,t)\geq \xi\omega\}\cap B_\rho(x_0)|\geq \alpha|B_\rho(x_0),\;\;\;\;\;\textup{ for all } t\in (t_0-\theta\rho^{sp},t_0].$$
    then either
    $$\frac{1}{\delta}\textup{Tail }[(u-\mu^\pm)_\pm;\mathcal{Q}]>\sigma\xi\omega$$
    or
    $$|\{\pm(\mu_\pm-u(\cdot,t)\geq \sigma\xi\omega\}\cap (z_0+\mathcal{Q}_\rho^\ominus(\theta))|\leq \gamma\frac{\sigma^{p-1}}{\delta\alpha}|z_0+\mathcal{Q}_\rho^\ominus(\theta)|,$$
\end{lemma}
provided $z_0+\mathcal{Q}_{2\rho}^\ominus(\theta)\subset\mathcal{Q}.$
\begin{proof}
    Without loss of generality, we consider $(x_0,t_0)=(0,0)$. We consider the truncated function
    \begin{align*}
        w_+=(u-2\sigma M)_+,
    \end{align*}
   where $\sigma M:=-\sigma\xi\omega+\mu^+$. Employ this function in energy estimate along with the $B_{2\rho}\times (-\theta\rho^{sp},0)$ where we choose $\theta=\delta(\sigma\xi\omega)^{2-p}$.  Consider the time-independent cut-off function of $\psi$ in $B_{2\rho}$ defined as follow
   \begin{align*}
       \psi\equiv \begin{cases}
           & 1\;\;\;\mbox{ on } B_\rho\\
           & 0\;\;\;\mbox{ on } B_{\frac{3}{2}\rho}
       \end{cases}
       \;\;\;\;\mbox{ such that } |\nabla_{\hn}\psi|\leq\frac{\gamma}{\rho}
   \end{align*}
   In this course of estimation, we can not ignore the {\em good} term. Using this cut-off function along with the considered cylinder, we have
   \begin{align}\label{eq3.9}
       \int_{-\delta(\sigma\xi\omega)^{2-p}\rho^{sp}}^{0}&\int_{B_\rho}w_+(x,t)dx\Bigg\{\int_{B_{2\rho}}\frac{w^{p-1}_-(y,t)}{\dist[x]{y}^{Q+sp}}dy\Bigg\}dt\nonumber\\
       &\leq \int_{-\delta(\sigma\xi\omega)^{2-p}\rho^{sp}}^{0}\intb[B]{2\rho}\intb[B]{2\rho}\max\{w^p_+(x,t),w^p_+(y,t)\}\gaglifrac[\psi]dxdydt\nonumber\\
       &\quad + \gamma\int_{-\delta(\sigma\xi\omega)^{2-p}\rho^{sp}}^{0}\intb[B]{2\rho}(w_+\psi^p)(x,t)\Bigg\{\tail[w_+]{x}{y}{2\rho}{B}dy\Bigg\}dxdt\\
       &\qquad +\intb[B]{2\rho}w^2_+(x,-\delta(\sigma\xi\omega)^{2-p}\rho^{sp})dx.\nonumber
   \end{align}
   Now successively estimate the right-hand side of (\ref{eq3.9}). The third term is estimated as follows
   \begin{align*}
       \intb[B]{2\rho}w^2_+(x,-\delta(\sigma\xi\omega)^{2-p}\rho^{sp})dx&\leq (\sigma\xi\omega)^2|B_{2\rho}|\\
       &\leq \gamma \frac{(\sigma\xi\omega)^p}{\delta\rho^{sp}}|\mathcal{Q}_\rho(\theta)|
   \end{align*}
   recalling that $|B_{2\rho}|= \mathcal{C}|B_\rho|$.\\
   The estimation of the first term on the right-hand side is below,
   \begin{align*}
       \int_{-\delta(\sigma\xi\omega)^{2-p}\rho^{sp}}^{0}&\intb[B]{2\rho}\intb[B]{2\rho}\max\{w^p_+(x,t),w^p_+(y,t)\}\gaglifrac[\psi]dxdydt\\
       &\leq \gamma\delta(\sigma\xi\omega)^2\rho^{sp-p}\intb[B]{2\rho}\intb[B]{2\rho}\frac{1}{\dist[x]{y}^{Q+(s-1)p}}dxdy\\
       &\leq \gamma\delta(\sigma\xi\omega)^2\rho^{sp-p}\frac{|B_{2\rho}|}{\rho^{(s-1)p}}\\
       &\leq \gamma\frac{(\sigma\xi\omega)^p}{\rho^{sp}}|\mathcal{Q}_\rho(\theta)|
   \end{align*}
Af for the {\em tail} estimate in the second term in (\ref{eq3.9}), note that $x\in \mbox{supp } \psi\subset B_{\frac{3}{2}\rho}$ and $y\in B_{2\rho}^c$, there holds
\begin{align*}
    \frac{\dist[x]{y}}{|y|_{\hn}}\geq 1-\frac{|x|_{\hn}}{|y|_{\hn}}\geq \frac{1}{4}
\end{align*}
Consequently,
\begin{align*}
    \int_{-\delta(\sigma\xi\omega)^{2-p}\rho^{sp}}^{0}&\intb[B]{2\rho}(w_+\psi^p)(x,t)\Bigg\{\tail[w_+]{x}{y}{2\rho}{B}dy\Bigg\}dxdt\\
    & \leq \gamma (\sigma\xi\omega)|B_{2\rho}|\int_{-\delta(\sigma\xi\omega)^{2-p}\rho^{sp}}^{0}\Bigg\{\int_{B_R\setminus B_{2\rho}}\frac{w_+^{p-1}(y,t)}{|y|_{\hn}^{Q+sp}}dy+\int_{\hn\setminus B_R}\frac{w_+^{p-1}(y,t)}{|y|_{\hn}^{Q+sp}}dy\Bigg\}dt\\
    &\leq \gamma (\sigma\xi\omega)|B_{2\rho}|\Bigg\{\gamma\delta(\sigma\xi\omega)+\int_{-\delta(\sigma\xi\omega)^{2-p}\rho^{sp}}^{0}\int_{\hn\setminus B_R}\frac{w_+^{p-1}(y,t)}{|y|_{\hn}^{Q+sp}}dy\Bigg\}dt\\
    &\leq \gamma\frac{(\sigma\xi\omega)^p}{\rho^{sp}}|\mathcal{Q}_\rho|
\end{align*}
In the course of calculations, the second inequality is due to \cite[Lemma 2.6]{MPPP23_3}. The last inequality is obtained by using the {\em tail} assumption for this lemma and $|B_{2\rho}|= \mathcal{C}|B_\rho|$. Combining the last three estimates into the right-hand side of the energy estimates, it becomes
\begin{align*}
    \int_{-\delta(\sigma\xi\omega)^{2-p}\rho^{sp}}^{0}\int_{B_\rho}w_+(x,t)dx\Bigg\{\int_{B_{2\rho}}\frac{w^{p-1}_-(y,t)}{\dist[x]{y}^{Q+sp}}dy\Bigg\}dt\leq \gamma\frac{(\sigma\xi\omega)^p}{\delta\rho^{sp}}|\mathcal{Q}_\rho|
\end{align*}
It remains to estimate the left-hand side of the previous inequality. This will be done by extending the integral over the smaller sets and by employing the given measure-theoretical information:
\begin{align*}
    \int_{-\delta(\sigma\xi\omega)^{2-p}\rho^{sp}}^{0}&\int_{B_\rho}w_+(x,t)dx\Bigg\{\int_{B_{2\rho}}\frac{w^{p-1}_-(y,t)}{\dist[x]{y}^{Q+sp}}dy\Bigg\}dt\\
    &\geq \int_{-\delta(\sigma\xi\omega)^{2-p}\rho^{sp}}^{0}\int_{B_\rho}w_+(x,t)\chi_{\{u(x,t)\geq \sigma M\}}dx\Bigg\{\intb[B]{2\rho}\frac{w^{p-1}_-(y,t)\chi_{\{u(x,t)\leq M\}}}{\dist[x]{y}^{Q+sp}}dy\Bigg\}dt\\
    &\geq \sigma\xi\omega|\{u(x,t)\geq \sigma M\}\cap\mathcal{Q}_\rho|\Bigg(\frac{(\frac{1}{4}\xi\omega)^{p-1}\alpha|B_\rho|}{(4\rho)^{Q+sp}}\Bigg)\\
    &\geq 4^{-(Q+p+sp-1)}\frac{(\xi\omega)^p\alpha\sigma}{\rho^{sp}}|\{u(x,t)\geq \sigma M\}\cap\mathcal{Q}_\rho|
\end{align*}
By combining these estimates and a suitable adjustment of the constants, we have our result.
\end{proof}
\section{Singular case: $1<p\leq 2$}\label{S4}
The following expansion of positivity, which is one of the main ingredients in the context of H\"older and Harnack estimates, translate positivity from measure estimate to pointwise one. It works as to supress the oscillation which plays a key role to derive our main result.  
\subsection{Expansion of Positivity}
\begin{lemma}\label{L4.1}
Consider some constants $\alpha, \xi\in (0,1)$ and $1<p\leq 2$. For locally bounded, local weak sub(super)solution to (\ref{main}) in $\Omega_T$, there exists constants $\delta,\eta\in (0,1),\tilde{\gamma}>1$ depending only on the \textup{\texttt{data}} and $\alpha$ such that if $B_{4\rho}(x_0)\times (t_0,t_0+\delta(\xi\omega)^{2-p}\rho^{sp}\subset\mathcal{Q}$ and 
$$|\{\pm(\mu^\pm-u(\cdot,t_0)\geq \xi\omega\}\cap B_\rho(x_0)|\geq \alpha|B_\rho(x_0)|.$$
Then either 
$$\tilde{\gamma}\textup{Tail }[(u-\mu^\pm)_\pm;\mathcal{Q}]>\eta\xi\omega$$
or
$$\pm(\mu^\pm-u)\geq \eta\xi\omega,\;\;\;\;\;\textup{ a.e. in } B_{2\rho}(x_0)\times(t_0+\frac{1}{2}\delta(\xi\omega)^{2-p}\rho^{sp},t_0+\delta(\xi\omega)^{2-p}\rho^{sp}].$$
Moreover, $\delta\approx\alpha^{p+Q+1}$ and $\eta\approx\alpha^q$ for some $q>1$ depending on the \textup{\texttt{data}}. 
\end{lemma}
\begin{proof}
    We consider proving this lemma for the subsolution case. The super counterpart follows the same way. Without loss of any generality, we consider $(x_0,t_0)=(0,0)$. Particularly, we want to show that there exist $\delta$ and $\eta$ such that
    \begin{align*}
        |\{u(\cdot,0)\leq -\xi\omega+\mu^+\}\cap B_\rho|\geq \alpha |B_\rho|
    \end{align*}
then either
\begin{align*}
    \frac{1}{\delta}\textup{Tail }[(u-\mu^\pm)_\pm;\mathcal{Q}]>\xi\omega
\end{align*}
or
\begin{align*}
    u(\cdot,0)\leq -\eta\xi\omega+\mu^+\;\;\;\mbox{a.e. in } B_{2\rho}\times(\frac{1}{2}\delta(\xi\omega)^{2-p}\rho^{sp},\delta(\xi\omega)^{2-p}\rho^{sp}].
\end{align*}
In Lemma \ref{L3.3}, we compute the measure-theoretical information in a forward time slice. Rewriting that results in a large ball $B_{4\rho}$ along with enforcing
\begin{align*}
   \frac{1}{\delta}\textup{Tail }[(u-\mu^\pm)_\pm;\mathcal{Q}]\leq \xi\omega
\end{align*}
We have for some $\delta,\varepsilon\in (0,1)$ depending only on \texttt{data} and $\alpha$, we get
\begin{align*}
    |\{u(\cdot,t)\leq -\varepsilon\xi\omega+\mu^+\}\cap B_{4\rho}|\geq \frac{\alpha}{2}\mathcal{C}^{-Q}|B_{4\rho}|,\;\;\;\forall\; t\in(0,\delta(\xi\omega)^{2-p}(4\rho)^{sp}]. 
\end{align*} 
This measure-theoretical information allows us to consider the shrinking lemma \ref{L3.4} in the cylinder $(0,\bar{t})+\mathcal{Q}^\ominus_{4\rho}(\delta(\xi\omega)^{2-p})$ with an arbitrary $\bar{t}\in (\delta(\sigma\varepsilon\xi\omega)^{2-p}(4\rho)^{sp},\delta(\xi\omega)^{2-p}(4\rho)^{sp}]$ with replacing $\xi$ and $\alpha$ by $\varepsilon\xi$ and $\frac{1}{2}\mathcal{C}^{-Q}\alpha$ respectively. With this choice of $\bar{t}$ and $\sigma\in (0,1)$, we have the following set inclusion i.e.
\begin{align*}
    (0,\bar{t})+\mathcal{Q}^\ominus_{4\rho}(\frac{1}{2}\delta(\sigma\varepsilon\xi\omega)^{2-p})\subset \mathcal{Q}^\oplus_{4\rho}(\delta(\xi\omega)^{2-p})
\end{align*}
Letting $\iota$ be determined in Lemma \ref{L3.1} in terms of \texttt{data} and $\delta$, further we choose $\sigma$ according to Lemma \ref{L3.4} to satisfy
\begin{align*}
    \gamma\frac{\sigma^{p-1}}{\delta\alpha}<\iota\;\;\mbox{ i.e. } \sigma<\Big(\frac{\iota\delta\alpha}{\gamma}\Big)^\frac{1}{p-1}.
\end{align*}
This choice is possible due to the independence of $\gamma$ on $\sigma$ in Lemma \ref{L3.4}. Letting $\tilde{\gamma}$ be chosen in Lemma \ref{L3.1} and further enforcing 
\begin{align*}
    \max\{4\mathfrak{C},\tilde{\gamma},\frac{1}{\delta}\}\textup{Tail }[(u-\mu^\pm)_\pm;\mathcal{Q}]\leq \sigma\varepsilon\xi\omega
\end{align*}
Due to this choice of $\sigma$, we are allowed to employ Lemma \ref{L3.4} and then Lemma \ref{L3.1} successively in the cylinder $(0,\bar{t})+\mathcal{Q}^\ominus_{4\rho}(\frac{1}{2}\delta(\sigma\varepsilon\xi\omega)^{2-p})$ with an arbitrary $\bar{t}$ as it is chosen earlier, and replacing $\xi$ by $\sigma\varepsilon\xi$, we conclude
\begin{align*}
    u\leq -\frac{1}{4}\sigma\varepsilon\xi\omega+\mu^+\;\;\;\mbox{ a.e. in } B_{2\rho}\times (\frac{1}{2}\delta(\sigma\varepsilon\xi\omega)^{2-p}(4\rho)^{sp},\delta(\xi\omega)^{2-p}(4\rho)^{sp}].
\end{align*}
Relabeling $\tilde{\gamma}$ for $\max\{4\mathfrak{C},\tilde{\gamma},1/\delta\}$ and defining $\eta=\frac{1}{4}\sigma\varepsilon$, we have the result.
\end{proof}
\subsection{Initial Step}\label{InStep}
Without loss of generality, we take $(x_0,t_0)=(0,0)$. Consider $\mathcal{R}>R$ such that $\mathcal{Q}^\ominus_\mathcal{R}\subset \Omega_T$. Introduce
\begin{align}\label{eq4.1}
    \omega\geq 2\esssup_{\mathcal{Q}^\ominus_\mathcal{R}}|u|+\mbox{Tail }(u;\mathcal{Q^\ominus_\mathcal{R}}),
\end{align}
and 
\begin{align*}
    \mathcal{Q}^\ominus_0:=\mathcal{Q}^\ominus_R(\omega^{2-p})
\end{align*}
By the proper choice of $\mathcal{R}$, $\mathcal{Q}^\ominus_0\subset\mathcal{Q}^\ominus_\mathcal{R}\subset \Omega_T$. Now set, 
\begin{align*}
    \mu^+:=\esssup_{\mathcal{Q}^\ominus_0}u,\;\;\; \mu^-:=\essinf_{\mathcal{Q}^\ominus_0} u.
\end{align*}
Then the following intrinsic relation is obvious from (\ref{eq4.1})
\begin{align}\label{eq4.2}
    \essosc_{\mathcal{Q}^\ominus_0}u\leq \omega
\end{align}
This is the starting induction argument to follow. Introduce a smaller cylinder
\begin{align*}
    \tilde{\mathcal{Q}}^\ominus_0:=\mathcal{Q}_{R}^\ominus(\omega^{2-p}c^{sp})\subset \mathcal{Q}^\ominus_0
\end{align*}
for some $c\in (0,1/4)$ to be chosen. Define
\begin{align*}
    \tau:=\delta (\frac{1}{4}\omega)^{2-p}(cR)^{sp}
\end{align*}
where $\delta\in(0,1)$ as in chosen in Lemma \ref{L4.1} with $\alpha=1/2$. \vskip 1mm
If $\mu^+-\mu^-<\frac{1}{2}\omega$, then it will be trivially incorporated in the forthcoming oscillation estimate. Otherwise, if $\mu^+-\mu^-\geq \frac{1}{2}\omega$ holds, then it welcomes one of the following two alternatives.
\begin{align*}
    \begin{cases}
        &|\{u(\cdot,-\tau)-\mu^->\frac{1}{4}\omega\}\cap B_{cR}|\geq \frac{1}{2}|B_{cR}|\\
        & |\{\mu^+-u(\cdot,-\tau)>\frac{1}{4}\omega\}\cap B_{cR}|\geq \frac{1}{2}|B_{cR}|
    \end{cases}
\end{align*}
Let us suppose the second alternative holds, for instance. Then by Lemma \ref{L4.1} with $\xi=\frac{1}{4}$, $\alpha=\frac{1}{2}$ and $\rho=cR$ determines $\eta\in (0,1/2)$, either 
\begin{align*}
    \tilde{\gamma} \mbox{ Tail } [(u-\mu^+)_+;\tilde{\mathcal{Q}}^\ominus_0]>\eta\omega
\end{align*}
or,
\begin{align*}
    \mu^+-u\geq \eta\omega \;\;\;\;\mbox{ a.e. in } \mathcal{Q}^\ominus_{cR}(\frac{1}{2}\delta (\frac{1}{4}\omega)^{2-p})
\end{align*}
In either case, taking (\ref{eq4.2}) into account, we have
\begin{align}\label{eq4.3}
    \essosc_{\mathcal{Q}^\ominus_{cR}(\frac{1}{2}\delta (\frac{1}{4}\omega)^{2-p})}\,u\leq \max\Big\{(1-\eta)\omega,\hat{\gamma}\,\mbox{Tail}[(u-\mu^+)_+;\tilde{\mathcal{Q}}^\ominus_0]\Big\}=:\omega_1
\end{align}
where $\hat{\gamma}:=\tilde{\gamma}/\eta$. Now it is unclear whether, due to the presence of the tail, $\omega_1$ is dominated by $\omega$ or not. Therefore, it is evident to estimate the tail term in such a way that a middle-distance term and a faraway tail term are individually controlled by $\omega$. So the tail estimate
\begin{align*}
  \mbox{Tail }[(u-\mu^+)_+;\mathcal{\tilde{\mathcal{Q}}}^\ominus_0]&=\int_{-\omega^{2-p}(cR)^{sp}}^0\int_{\hn\setminus B_R}\frac{(u-\mu^+)_+^{p-1}}{|y|_{\hn}^{Q+sp}}dydt \\
  &\leq \gamma c^{sp}\omega+ \gamma\int_{-\omega^{2-p}(cR)^{sp}}^0\int_{\hn\setminus B_R}\frac{u_+^{p-1}}{|y|_{\hn}^{Q+sp}}dydt\\
  &= \gamma c^{sp}\omega+ \gamma\int_{-\omega^{2-p}(cR)^{sp}}^0\int_{B_\mathcal{R}\setminus B_R}\frac{u_+^{p-1}}{|y|_{\hn}^{Q+sp}}dydt+\gamma \int_{-\omega^{2-p}(cR)^{sp}}^0\int_{\hn\setminus B_\mathcal{R}}\frac{u_+^{p-1}}{|y|_{\hn}^{Q+sp}}dydt\\
  &\leq \gamma c^{sp}\omega+\gamma \int_{-\omega^{2-p}(cR)^{sp}}^0\int_{\hn\setminus B_\mathcal{R}}\frac{u_+^{p-1}}{|y|_{\hn}^{Q+sp}}dydt\\
  &\leq \gamma \omega.
\end{align*}
In this event of calculations, first and last inequalities come from (\ref{eq4.2}). The second last inequality is due to Lemma \ref{L1.1}. So, we obtain 
\begin{align*}
    \mbox{Tail }[(u-\mu^+)_+;\tilde{\mathcal{Q}}^\ominus_0]\leq \gamma\omega,\;\;\;\;\omega_1\leq \bar{\gamma}\omega
\end{align*}
Next, we introduce $R_1=\lambda R$ for some $\lambda\leq c$ to verify the set inclusion
\begin{align}\label{eq4.4}
    \mathcal{Q}^\ominus_{R_1}(\omega_1^{2-p})\subset \mathcal{Q}^\ominus_{cR}(\frac{1}{2}\delta (\frac{1}{4}\omega)^{2-p})\;\;\;\;\mbox{i.e.},\; \lambda\leq 2^\frac{2p-5}{sp}\Bigg(\frac{1}{1-\eta}\Bigg)^\frac{p-2}{sp}\delta^\frac{1}{sp}c.
\end{align}
As a consequence of this result, along with (\ref{eq4.3}), we have
\begin{align*}
    \mathcal{Q}^\ominus_{R_1}(\omega_1^{2-p})\subset \mathcal{Q}^\ominus_0\;\;\mbox{ and } \esssup_{\mathcal{Q}^\ominus_{R_1}(\omega_1^{2-p})}u\leq \omega_1,
\end{align*}
which plays the role of (\ref{eq4.2}) in the very next stage. At this stage, $c\in(0,1/4)$ is still to be chosen. Note that the calculations will be similar if the first alternative holds instead. In that very case, one needs to replace $(u-\mu^+)_+$ with $(u-\mu^-)_-$ only.  
\subsection{Induction Step} As we have done the primary process in the previous step, we can proceed by induction. Suppose up to $i=1,2,3,...,j$, we have set
\begin{align*}
    \begin{cases}
        &R_0=R,\;\;R_i=\lambda R_{i-1},\;\;(1-\eta)\omega_{i-1}\leq\omega_i\leq \bar{\gamma}\omega_{i-1},\\
        &\omega_0=\omega,\;\;\omega_i=\max\{(1-\eta)\omega_{i-1},\;\;\hat{\gamma}\mbox{ Tail }[(u-\mu^\pm_{i-1})_\pm;\tilde{Q}^\ominus_{i-1}\},\\
        & \mathcal{Q}_i^\ominus=\mathcal{Q}_{R_i}^\ominus(\omega_i^{2-p}),\;\; \tilde{\mathcal{Q}}_i^\ominus=\mathcal{Q}_{R_i}^\ominus(\omega_i^{2-p}c^{sp}),\;\;\mathcal{Q}_i^\ominus\subset\mathcal{Q}^\ominus_{i-1},\\
        &\mu^+_i=\esssup_{\mathcal{Q}^\ominus_i}\;u,\;\;\mu^-_i=\essinf_{\mathcal{Q}^\ominus_i}\;u,\;\;\essosc_{\mathcal{Q}^\ominus_i}\leq \omega_i
    \end{cases}
\end{align*}
The main motive of this subsection is to reduce the oscillation in the next stage. For this purpose, we mainly replicate the process of the Initial Step (\autoref{InStep}) with $\mu_j^\pm,\omega_j, R_j,\mathcal{Q}^\ominus_j, \tilde{\mathcal{Q}}^\ominus_j$. This induction estimate will help the above oscillation continue to hold for the $(j+1)$-th step. Let $\delta$ is fixed and $c\in (0,1)$ to be chosen later. Define
\begin{align*}
    \tau:=\delta(\frac{1}{4}\omega_j)^{2-p}(cR_j)^{sp}
\end{align*}
and consider two alternatives
\begin{align}\label{eq4.5}
    \begin{cases}
        &|\{u(\cdot,-\tau)-\mu_j^->\frac{1}{4}\omega_j\}\cap B_{cR_j}|\geq \frac{1}{2}|B_{cR_j}|\\
        &|\{\mu_j^+-u(\cdot,-\tau)>\frac{1}{4}\omega_j\}\cap B_{cR_j}|\geq \frac{1}{2}|B_{cR_j}|
    \end{cases}
\end{align}
As the initial step, we may assume $\mu_j^+-\mu^-_j\geq \frac{1}{2}\omega_j$, so any one of the alternatives will be true. otherwise, if in the case $\mu_j^+-\mu^-_j< \frac{1}{2}\omega_j$, then it can incorporate into the forthcoming oscillation estimate.\vskip 1mm
Let us suppose that the second alternative holds true. Then apply Lemma \ref{L4.1} with $\mathcal{Q}^\ominus_j$ with $\alpha=1/2,\xi=1/4, \rho=cR_j$ and same $\eta$ as before yields that, either
\begin{align*}
    \tilde{\gamma} \mbox{ Tail }[(u-\mu^+_j)_+;\tilde{\mathcal{Q}}^\ominus_j]>\eta\omega_j.
\end{align*}
or
\begin{align*}
    \mu^+_j-u\geq \eta\omega_j\;\;\;\mbox{a.e. in } \mathcal{Q}_{cR_j}^\ominus(\frac{1}{2}\delta(\frac{1}{4}\omega_j)^{2-p}),
\end{align*}
which gives, thanks to the $j$th induction assumption, i.e $\essosc_{\mathcal{Q}^\ominus_j(\omega_j^{2-p})}u\leq \omega_j$,
\begin{align}\label{eq4.6}
  \essosc_{\mathcal{Q}^\ominus_{cR_j}(\frac{1}{2}\delta(\frac{1}{4}\omega_j)^{2-p})}u\leq \max\Big\{(1-\eta)\omega_j,\hat{\gamma}\,\mbox{Tail}[(u-\mu^+_j)_+;\tilde{\mathcal{Q}}^\ominus_j]\Big\}=:\omega_{j+1},
\end{align}
where $\hat{\gamma}$ as in (\ref{eq4.3}). \vskip 1mm
Here we create the same type of moment as (\ref{eq4.3}), where it is unclear, due to the presence of {\em tail}, whether $\omega_{j+1}$ is controlled by $\omega_j$. For induction purpose it is sufficient to show $\omega_{j+1}\leq \bar{\gamma}\omega_j$ for some $\bar{\gamma}>1$. But this $\bar{\gamma}$ depends on the tail. To make sure $\bar{\gamma}$ should only depend on \texttt{data} and independent of $j$, one need to estimate the tail term as follows:
\begin{align}\label{eq4.7}
    \mbox{ Tail }[(u-\mu^+_j)_+;\tilde{\mathcal{Q}}_j^\ominus]&=\int_{-\omega_j^{2-p}(cR_j)^{sp}}^0\intb[\hn\setminus B_R]{j}\frac{(u-\mu_j^+)_+^{p-1}}{|x|_{\hn}^{Q+sp}}dxdt\nonumber\\
    &=\underbrace{\int_{-\omega_j^{2-p}(cR_j)^{sp}}^0\intb[\hn\setminus B]{R}\frac{(u-\mu_j^+)_+^{p-1}}{|x|_{\hn}^{Q+sp}}dxdt}_{I}\\
    &\quad+\sum_{i=1}^j\underbrace{\int_{-\omega_j^{2-p}(cR_j)^{sp}}^0\intb[B_{R_{j-1}}\setminus B_R]{j}\frac{(u-\mu_j^+)_+^{p-1}}{|x|_{\hn}^{Q+sp}}dxdt}_{II}\nonumber
\end{align}
We estimate $I$ and $II$ separately. 
\begin{align*}
    I&\leq \gamma \omega_j^{2-p}(cR_j)^{sp}\frac{\omega^{p-1}}{R^{sp}}+ \int_{-\omega_j^{2-p}(cR_j)^{sp}}^0\gamma \int_{B_{\mathcal{R}}\setminus B_R}\frac{u^{p-1}_+}{|x|^{Q+sp}_{\hn}}dxdt + \gamma \int_{-\omega_j^{2-p}(cR_j)^{sp}}^0\intb[\hn\setminus B]{\mathcal{R}} \frac{u^{p-1}_+}{|x|^{Q+sp}_{\hn}}dxdt \\
    & \leq \gamma \omega_j^{2-p}(cR_j)^{sp}\frac{\omega^{p-1}}{R^{sp}}+ \gamma \int_{-\omega_j^{2-p}(cR_j)^{sp}}^0\intb[\hn\setminus B]{\mathcal{R}} \frac{u^{p-1}_+}{|x|^{Q+sp}_{\hn}}dxdt.
\end{align*}
In this sequence of calculations, we use $|\mu^+|\leq \omega$ on $\mathcal{Q}^\ominus_\mathcal{R}$ in the first inequality, whereas in the last inequality we use $u_+\leq \omega$ on  $\mathcal{Q}^\ominus_\mathcal{R}$.\vskip 1mm
For $i=1,2,3,...,j$,
\begin{align*}
    (u-\mu_j^+)_+=(\mu^+_j-u)_-\leq \mu^+_j-\mu^-_{i-1}\leq \mu^+_{i-1}-\mu^-_{i-1}\leq \omega_{i-1}\,\,\,\mbox{ a.e. in } \mathcal{Q}^\ominus_{i-1}.
\end{align*}
Consequently, $II$ implies,
\begin{align*}
    II\leq \gamma \omega_j^{2-p}(cR_j)^{sp}\frac{\omega_{i-1}^{p-1}}{R_i^{sp}}
\end{align*}
Combining estimations of $I$ and $II$ into (\ref{eq4.7}), we have the following
\begin{align}\label{eq4.8}
    \mbox{ Tail }[(u-\mu^+_j)_+;\tilde{\mathcal{Q}}^\ominus_j]&\leq \gamma \omega_j^{2-p}(cR_j)^{sp}\frac{\omega^{p-1}}{R^{sp}}+ \gamma \underbrace{\sum_{i=1}^j\omega_j^{2-p}(cR_j)^{sp}\frac{\omega_{i-1}^{p-1}}{R_i^{sp}}}_{III}\nonumber\\
   &\quad +\gamma \underbrace{\int_{-\omega_j^{2-p}(cR_j)^{sp}}^0\intb[\hn\setminus B]{\mathcal{R}} \frac{u^{p-1}_+}{|x|^{Q+sp}_{\hn}}dxdt}_{IV}.
\end{align}
Consider the first two terms on the right-hand side of (\ref{eq4.8}). To estimate these terms, let us recall the definitions of $\omega_j$ and $R_j$, defined at the beginning of the induction, such that for $i=1,2,3,...,j$,
\begin{align}\label{eq4.9}
    (1-\eta)^{j-i}\omega_i\leq\omega_j,\,\,\, R_j=\lambda^{j-i}R_i.
\end{align}
Using the outcomes in (\ref{eq4.9}), $III$ implies,
\begin{align*}
    III\leq \gamma\omega_j\sum_{i=1}^j((1-\eta)^{1-p}c^{sp})^{j-i+1}.
\end{align*}
In this inequality, we used the fact that $\lambda\leq c$. Now we are left with the estimate of $IV$, which can be done as follows using (\ref{eq4.9}) and definition of tail,
\begin{align*}
    IV&\leq \int_{-\omega_j^{2-p}(cR_j)^{sp}}^0\intb[\hn\setminus B]{\mathcal{R}} \frac{(u-\mu^+_{j-1})^{p-1}_++|\mu^+_{j-1}|^{p-1}}{|x|^{Q+sp}_{\hn}}dxdt\\
    &\leq \gamma\omega_j^{2-p}(cR_j)^{sp}\frac{\omega^{p-1}}{R^{sp}}+\int_{-\omega_j^{2-p}(cR_j)^{sp}}^0\intb[\hn\setminus B]{\mathcal{R}} \frac{(u-\mu^+_{j-1})^{p-1}_+}{|x|^{Q+sp}_{\hn}}dxdt\\
    &\leq \gamma\omega_j((1-\eta)^{1-p}c^{sp})^j+\gamma\mbox{ Tail }[(u-\mu^+_{j-1})_-;\tilde{\mathcal{Q}}^\ominus_{j-1}]\\
    &\leq \gamma\omega_j((1-\eta)^{1-p}c^{sp})^j+\frac{\gamma}{\hat{\gamma}}\omega_j
\end{align*}
Combining estimates of $III$ and $IV$ into (4.8), we obtain
\begin{align}\label{eq4.10}
    \mbox{ Tail }[(u-\mu^+_j)_+;\tilde{\mathcal{Q}}^\ominus_j]\leq \gamma\omega_j\sum_{i=1}^j((1-\eta)^{1-p}c^{sp})^{j-i+1}+\frac{\gamma}{\hat{\gamma}}\omega_j.
\end{align}
The summation on the right-hand side of (\ref{eq4.10}) is bounded by if 
\begin{align}\label{eq4.11}
    (1-\eta)^{1-p}c^{sp}<\frac{1}{2}\,\,\,\mbox{ i.e., }\;\; c<\Bigg(\frac{(1-\eta)^{p-1}}{2}\Bigg)^{1/sp}
\end{align}
Therefore, we have
\begin{align}\label{eq4.12*}
    \mbox{ Tail }[(u-\mu_j^+)_+;\tilde{\mathcal{Q}}^\ominus_j]\leq \gamma\omega_j,\;\;\omega_{j+1}\leq \bar{\gamma}\omega_j.
\end{align}
Here, $\bar{\gamma}$ can be different from the previous one in the Initial Step.\vskip 1mm
Let $R_{j+1}=\lambda R_j$ for some $\lambda\in (0,1)$ to verify the set inclusion as done in (\ref{eq4.4})
\begin{align}\label{eq4.12}
    \mathcal{Q}^\ominus_{R_{j+1}}(\omega_{j+1}^{2-p})\subset \mathcal{Q}^\ominus_{cR_j}(\frac{1}{2}\delta (\frac{1}{4}\omega_j)^{2-p})\;\;\;\;\mbox{i.e.},\; \lambda\leq 2^\frac{2p-5}{sp}\Bigg(\frac{1}{1-\eta}\Bigg)^\frac{p-2}{sp}\delta^\frac{1}{sp}c
\end{align}
As a conclusion from the results in (\ref{eq4.12}) and (\ref{eq4.6}), we obtain
\begin{align*}
    \mathcal{Q}^\ominus_{R_{j+1}}(\omega_{j+1}^{2-p})\subset\mathcal{Q}^\ominus_{R_j}(\omega_j^{2-p})\;\;\mbox{ and } \essosc_{\mathcal{Q}^\ominus_{R_{j+1}}(\omega_{j+1}^{2-p})} u\leq \omega_{j+1},
\end{align*}
that concludes the induction argument for the second alternative. The process will be the same if the first alternative in (\ref{eq4.5}) holds. For that case, one needs to proceed with $(u-\mu^-_j)_-$, instead of $(u-\mu^+_j)_+$. Note that the smallness assumption imposed on $c$ in (\ref{eq4.11}) is not the final choice and may change according to the requirements in the sequel.
\subsection{Modulus of Continuity for Singular case}\label{sub4.4} From the previous subsection, concluding through the induction arguments, we obtain the oscillation estimate
\begin{align}\label{eq4.13}
    \essosc_{\mathcal{Q}^\ominus_n}\leq \omega_n=\max\{(1-\eta)\omega_{n-1},\tilde{\gamma}\mbox{ Tail }[(u-\mu^\pm_{n-1})_\pm;\tilde{\mathcal{Q}}^\ominus_{n-1}\},\;\;\;n\in\mathbb{N}
\end{align}
along with the inclusion $\mathcal{Q}^\ominus_n\subset\mathcal{Q}^\ominus_{n-1}$, where $R_n=\lambda^nR,\,\mathcal{Q}^\ominus_n=\mathcal{Q}^\ominus(\omega_n^{2-p})$ for $n\in\mathbb{N}$.\vskip 1mm
Keeping (\ref{eq4.13}) in mind, we consider (\ref{eq4.8}) and (\ref{eq4.10}) (with $j=n-1$). Indeed, we may further need the smallness of $c$, i.e.
\begin{align}\label{eq4.14}
    \gamma\omega_{n-1}\sum_{i=1}^{n-1}((1-\eta)^{1-p}c^{sp})^{n-i}\leq \frac{1}{2\hat{\gamma}}\omega_{n-1}\,\,\,\mbox{ i.e. }\,\,\,c\leq \Bigg(\frac{(1-\eta)^{p-1}}{2\gamma\hat{\gamma}}\Bigg)^\frac{1}{sp}.
\end{align}
Hence, combining this estimate with the definition of $\omega_n$ in (\ref{eq4.13}), we have that
\begin{align*}
    \omega_n&\leq (1-\eta)\omega_{n-1}+\gamma\int_{-\omega_{n-1}^{2-p}(cR_{n-1})^{sp}}^0{\int_{{\hn}\setminus B_\mathcal{R}}}\frac{|u|^{p-1}}{|x|_{\hn}^{Q+sp}}\,dxdt\\
    &\leq (1-\eta)\omega_{n-1}+\gamma\int_{-\omega^{2-p}(cR)^{sp}}^0{\int_{{\hn}\setminus B_\mathcal{R}}}\frac{|u|^{p-1}}{|x|_{\hn}^{Q+sp}}\,dxdt.
\end{align*}
Iterating the above result, we have the following result
\begin{align}\label{eq4.15}
    \essosc_{\mathcal{Q}^\ominus_n}u\leq \omega_n\leq (1-\eta)^n\omega+\gamma\int_{-\omega^{2-p}R^{sp}}^0\intb[\hn\setminus B]{\mathcal{R}}\frac{|u|^{p-1}}{|x|^{Q+sp}_{\hn}}dxdt.
\end{align}
From the result in (\ref{eq4.15}) and assumption in (\ref{eq4.1}), we can obtain $\omega_n\leq \gamma^*\omega$ for taking $\gamma^*=1+\gamma$. \vskip 1mm 
Since, $\{\omega_n^{2-p}R_n^{sp}\}$ is decreasing sequence, tending to $0$, then for some fixed $r\in (0,R)$, there exist some $n\in\mathbb{N}$ yields that
\begin{align*}
    \omega_{n+1}^{2-p}R^{sp}_{n+1}\leq \omega^{2-p}r^{sp}<\omega_n^{2-p}R^{sp}_n.
\end{align*}
The right-hand side, along with the inequality $\omega_n\leq \gamma^*\omega$, implies 
\begin{align*}
    \sigma r<R_n\;\;\;\mbox{ and }\;\;\; \mathcal{Q}^\ominus_{\sigma r}(\omega^{2-p)}\subset\mathcal{Q}^\ominus_n\;\;\;\mbox{ for } \sigma={(\gamma^*)}^{{2-p}/{sp}},
\end{align*}
whereas the left-hand side inequality, along with the iterative term $(1-\eta)\omega_n\leq\omega_{n+1}$, such that
\begin{align*}
    \omega^{2-p}r^{sp}\geq  \omega_{n+1}^{2-p}R^{sp}_{n+1}\geq ((1-\eta)^{2-p}\lambda^{sp})^{n+1}\omega^{2-p}R^{sp}
\end{align*}
consequently,
\begin{align*}
    ((1-\eta)^{2-p}\lambda^{sp})^{n+1}\leq (r/R)^{sp}
\end{align*}
this implies
\begin{align}
    (1-\eta)^{n+1}\leq (r/R)^\mathfrak{b},\;\;\;\mbox{ where } \mathfrak{b}:=\frac{sp\mbox{ ln}(1-\eta)}{\mbox{ ln}((1-\eta)^{p-2}\lambda^{sp})}
\end{align}
Note that, depending upon the choice of a smaller $c$ between (\ref{eq4.11}) and (\ref{eq4.14}), we shall consider the smaller $\lambda$ from (\ref{eq4.12}). \vskip 1mm
Collecting all the estimates and inserting them into (\ref{eq4.15}), we have
\begin{align}\label{eq4.17}
    \essosc_{\mathcal{Q}^\ominus_{\sigma r}(\omega^{2-p})}u\leq 2\omega(r/R)^\mathfrak{b}+\gamma\int_{-\omega^{2-p}R^{sp}}^0\intb[\hn\setminus B]{\mathcal{R}}\frac{|u|^{p-1}}{|x|^{Q+sp}_{\hn}}dxdt
\end{align}
\section{Degenerate Case: $p>2$}\label{S5}
\subsection{Initial Step}
Without loss of generality, we take $(x_0,t_0)=(0,0)$. Consider $\mathcal{R}>R$ such that $\mathcal{Q}^\ominus_\mathcal{R}\subset\Omega_T$. Introduce
\begin{align}\label{eq5.1}
    \omega\geq 2\esssup_{\mathcal{Q}^\ominus_\mathcal{R}}|u|+\mbox{ Tail } (u;\mathcal{Q}^\ominus_\mathcal{R})
\end{align}
and
\begin{align*}
    \mathcal{Q}^\ominus_0:=\mathcal{Q}_R^\ominus(\mathcal{L}\theta),
\end{align*}
where $\mathcal{L}>R$ and $\theta:=(\frac{1}{4}\omega)^{2-p}$. By the proper choice of $\mathcal{R}$, $\mathcal{Q}^\ominus_0\subset\mathcal{Q}_\mathcal{R}\subset\Omega_T$. Set,
\begin{align*}
    \mu^+:=\esssup_{\mathcal{Q}^\ominus_0}u,\;\;\; \mu^-:=\essinf_{\mathcal{Q}^\ominus_0} u.
\end{align*}
The  following intrinsic relation is obvious from (\ref{eq5.1})
\begin{align}
    \esssup_{\mathcal{Q}^\ominus_0}\leq \omega.
\end{align}
This is the very inequality where the starting argument of induction takes place. Introduce a smaller cylinder
\begin{align*}
    \tilde{\mathcal{Q}}_0^\ominus:=\mathcal{Q}^\ominus_R(\mathcal{L}\theta c^{sp})\subset\mathcal{Q}_0^\ominus
\end{align*}
for some $c\in (0,1/4)$. The numbers $c$ and $\mathcal{L}$ will be chosen in terms of \texttt{data}. \vskip 1mm
The further process in this section is slightly different from the previous section. Here, we unfold two alternatives in two different subsections. as done in \cite[Sections 5.2 and 5.2]{Lia24_2}. The concept of underlying intrinsic scaling is derived from DiBenedetto's work on the Parabolic $p$-Laplace Equation \cite[Chapter III]{DiB93_4}.
\subsection{The First Alternative}
In this section, we mainly deal with $u$ as a supersolution near its infimum.  We consider, without loss of any generality
\begin{align}\label{eq5.3}
    \mu^+-\mu^->\frac{1}{2}\omega.
\end{align}
The other case is the trivial argument of oscillation reduction. \vskip 1mm
Let we consider $\bar{t}\in (-(\mathcal{L}-1)\theta(cR)^{sp},0]$ and at that time level suppose the following holds
\begin{align}\label{eq5.4}
    |\{u\leq \mu^-+\frac{1}{4}\omega\}\cap(0,\bar{t})+\mathcal{Q}_{cR}^\ominus(\theta)|\leq \nu|\mathcal{Q}_{cR}^\ominus(\theta)|,
\end{align}
where the constant $\nu$ is chosen as in De Giorgi Lemma \ref{L3.1} (with $\delta=1$). Now using Lemma \ref{L3.1} with $\delta=1, \rho=cR$ and $\xi=1/4$, we have either
\begin{align}\label{eq5.5}
    \tilde{\gamma}\;\mbox{ Tail }[(u-\mu^-)_-;\tilde{\mathcal{Q}}^\ominus_0]>\frac{\omega}{4}
\end{align}
or
\begin{align}\label{eq5.6}
    u\geq \mu^-+\frac{\omega}{8}\;\;\;\mbox{ a.e. in } (0,\bar{t})+\mathcal{Q}^\ominus_{\frac{1}{2}cR}(\theta)
\end{align}
Next, we want to exploit the fact at the time level, 
\begin{align*}
    t^*=\bar{t}-\theta(\frac{1}{2}cR)^{sp}
\end{align*}
$u$ is above the level $\mu^-+1/4$ pointwise in the cube $B_\frac{cR}{4}$. For this purpose, considering the fact in (\ref{eq5.6}), use the Lemma \ref{L3.2} with $\rho=\frac{1}{2}cR$ to obtain that, for some arbitrary parameter $\xi_0\in (0,1/8)$, either
\begin{align}\label{eq5.7}
    \tilde{\gamma}\;\mbox{ Tail }[(u-\mu^-)_-;\tilde{\mathcal{Q}}^\ominus_0]>\xi_o\omega
\end{align}
or
\begin{align}\label{eq5.8}
    u\geq \mu^-+\frac{\xi_0\omega}{4}\;\;\;\mbox{ a.e. in } B_{\frac{cR}{4}}\times(t^*,t^*+\nu_0(\xi_0\omega)^{2-p}(\frac{1}{2}cR)^{sp}].
\end{align}
where we choose the numbers $\xi_0$ to satisfy
\begin{align}
    \nu_0(\xi_0\omega)^{2-p}(\frac{1}{2}cR)^{sp}\geq \mathcal{L}(\frac{1}{4}\omega)^{2-p}(cR)^{sp}\implies\xi_o\geq \frac{1}{4}\Big(\frac{\nu_0}{2^{sp}\mathcal{L}}\Big)^\frac{1}{p-2}
\end{align}
If both (\ref{eq5.5}) and (\ref{eq5.7}) do not hold, then the previous inequality implies that (\ref{eq5.8}) holds up to $t=0$ and yields that
\begin{align}\label{eq5.10}
    \essosc_{\mathcal{Q}_{\frac{cR}{4}}^\ominus(\theta)}\leq (1-\xi_0/4)\omega.
\end{align}
Note that (\ref{eq5.10}) is true while (\ref{eq5.5}) and (\ref{eq5.7}) do not true. In contrast, if one of them occurs, then such case we need to take in consideration in forthcoming oscillation estimate (\ref{eq5.21}). Moreover, $\mathcal{L}$ is yet to determined in the form \texttt{data}. 
\subsection{The Second Alternative}
This subsection is devoted to dealing with $u$ as a subsolution near its supremum. We assume that condition (\ref{eq5.4}) is violated, i.e. if the infimum of $u$ is tending nearer to its infimum, then for any time level $\bar{t}\in (-(\mathcal{L}-1)\theta(cR)^{sp},0]$, the following holds
\begin{align*}
    |\{u\geq \mu^+-\frac{1}{4}\omega\}\cap(0,\bar{t})+\mathcal{Q}_{cR}^\ominus(\theta)|\leq (1-\nu)|\mathcal{Q}_{cR}^\ominus(\theta)|
\end{align*}
that is
\begin{align}\label{eq5.11}
     |\{u\leq \mu^+-\frac{1}{4}\omega\}\cap(0,\bar{t})+\mathcal{Q}_{cR}^\ominus(\theta)|> \nu|\mathcal{Q}_{cR}^\ominus(\theta)|
\end{align}
This happened due to (\ref{eq5.3}).\vskip 1mm
In the view of (\ref{eq5.11}), we want to measure the portion of the ball $B_{cR}$ where $u$ moves away from its supremum in a certain time level. For this purpose, we can have some $t^*\in [\bar{t}-\theta(cR)^{sp},\bar{t}-\frac{1}{2}\nu\theta(cR)^{sp}]$ satisfying
\begin{align}\label{eq5.12}
    |\{u(\cdot,t^*)\leq \mu^+-\frac{1}{4}\omega\}\cap B_{cR}|> \frac{1}{2}\nu|B_{cR}|
\end{align}
If (\ref{eq5.12}) is not true, then for any time level $t^*$
\begin{align*}
    |\{u\leq \mu^+-\frac{1}{4}\omega\}\cap(0,\bar{t})+\mathcal{Q}_{cR}^\ominus(\theta)|&=\int_{\bar{t}-\theta(cR)^{sp}}^{\bar{t}-\frac{1}{2}\nu\theta(cR)^{sp}}|\{u(\cdot,s)\leq \mu^+-\frac{1}{4}\omega\}\cap B_{cR}|\;ds\\
    &\quad +\int_{\bar{t}-\frac{1}{2}\nu\theta(cR)^{sp}}^{\bar{t}} |\{u(\cdot,s)\leq \mu^+-\frac{1}{4}\omega\}\cap B_{cR}|\;ds\\
    &<\frac{1}{2}\nu|B_{cR}|\theta(cR)^{sp}(1-\frac{1}{2}\nu)+\frac{1}{2}\nu\theta(cR)^{sp}|B_{cR}|\\
    &<\nu|\mathcal{Q}^\ominus_{cR}(\theta)|,
\end{align*}
which creates a contradiction. \vskip 1mm
The measure theoretical information in (\ref{eq5.12}) asserts that $u$ moves away from its supremum in a significant portion of $B_{cR}$ at some certain time level $t^*$. Next, we demonstrate that it indeed occurs within a specific time interval. For this reason, using the result in (\ref{eq5.12}), we may apply Lemma \ref{L3.3} with $\alpha=\nu/2$ and $\rho=cR$, such that free parameter $\xi_1\in(0,1/4)$, either
\begin{align}\label{eq5.13}
    \tilde{\gamma}\mbox{ Tail }[(u-\mu^+)_+;\tilde{\mathcal{Q}}^\ominus_0]>\xi_1\omega
\end{align}
or
\begin{align}\label{eq5.14}
    |\{\mu^+-u(\cdot,t)\geq\varepsilon\xi_1\omega\}\cap B_{cR}|\geq \frac{\alpha}{2}|B_{cR}|\;\;\mbox{ for all }\;\; t\in(t^*,t^*+\delta(\xi_1\omega)^{2-p}(cR)^{sp}]  
\end{align}
where $\delta$ and $\varepsilon$ is depended on \texttt{data} and $\nu$ and $\xi_1$ is chosen to satisfy
\begin{align*}
    \delta(\xi_1\omega)^{2-p}(cR)^{sp}\geq \theta(cR)^{sp},\;\;\mbox{ that is },\;\;\xi_1\geq \frac{1}{4}\delta^\frac{1}{p-2}.
\end{align*}
This choice ensures that (\ref{eq5.14}) holds at all time levels near the top of the cylinder, i.e. (\ref{eq5.14}) yields 
\begin{align}\label{eq5.15}
    |\{\mu^+-u(\cdot,t)\geq \epsilon\xi_1\omega\}\cap B_{cR}|\geq \frac{\alpha}{2}|B_{cR}|\;\;\mbox{ for all }\;\;t\in (-(\mathcal{L}-1)\theta(cR)^{sp},0],
\end{align}
Using the result in (\ref{eq5.15}), we want to use the Shrinking Lemma \ref{L3.4} with $\xi=\varepsilon\xi_1,\delta=1$ and $\rho=cR$ next. First fix $\nu$ as in Lemma \ref{L3.1} with $\delta=1$ and then choose $\sigma\in (0,1/2)$ to verify
\begin{align*}
    \gamma\frac{\sigma^{p-1}}{\alpha}\leq \nu
\end{align*}
This choice is true because $\lambda$ of Lemma \ref{L3.4} is independent of $\sigma$. Then, L is determined by
\begin{align}\label{eq5.16}
    (\mathcal{L}-1)\theta(cR)^{sp}\geq (\sigma\varepsilon\xi_1\omega)^{2-p}(cR)^{sp},\;\;\;\mbox{ that is },\;\;\; \mathcal{L}\geq 1+(4\sigma\varepsilon\xi_1)^{2-p}.
\end{align}
where $\mathcal{L}$ is depending on \texttt{data} as $\sigma,\varepsilon,\xi_1$ are depending on \texttt{data}. So, the result in (\ref{eq5.15}) becomes 
\begin{align*}
    |\{\mu^+-u(\cdot,t)\geq \epsilon\xi_1\omega\}\cap B_{cR}|\geq \alpha|B_{cR}|\;\;\mbox{ for all }\;\;t\in (-\Theta (cR)^{sp},0],
\end{align*}
where $\Theta:= (\sigma\varepsilon\xi_1\omega)^{2-p}$. This inequality allows us to use the Lemma \ref{L3.4} that is 
\begin{align}\label{eq5.17}
    \tilde{\gamma}[\mbox{ Tail } (u-\mu^+)_+; \tilde{\mathcal{Q}}_0^\ominus]>\sigma\varepsilon\xi_1\omega
\end{align}
or
\begin{align}\label{eq5.18}
    |\{\mu^+-u\leq \sigma\varepsilon\xi_1\omega\}\cap\mathcal{Q}_{cR}^\ominus(\Theta)|\leq \nu|\mathcal{Q}_{cR}^\ominus(\Theta)|.
\end{align}
This inequality (\ref{eq5.18}) allows us to use Lemma \ref{L3.1} with $\delta=1$ and then we have
\begin{align*}
    \mu^+-u\geq \frac{1}{4}\sigma\varepsilon\xi_1\omega\;\;\;\mbox{ a.e. in } \mathcal{Q}^\ominus_{\frac{cR}{2}}(\Theta)
\end{align*}
which gives the reduction of oscillation
\begin{align}\label{eq5.19}
    \essosc_{\mathcal{Q}^\ominus_{\frac{cR}{2}}(\Theta)}u\leq (1-\frac{1}{4}\sigma\varepsilon\xi_1)\omega.
\end{align}\vskip 6mm
Assuming (\ref{eq5.5}, \ref{eq5.7}, \ref{eq5.13}) and (\ref{eq5.17}) not being true, we have the reduction of oscillation by combining (\ref{eq5.10}) and (\ref{eq5.19}) that is 
\begin{align}
    \essosc_{\mathcal{Q}^\ominus_{\frac{cR}{4}}(\theta)}\leq (1-\eta)\omega,
\end{align}
where $\eta=\min\{\frac{\epsilon_0}{4},\frac{1}{4}\sigma\varepsilon\xi_1\}$. Therefore, combining all the cases, we have
\begin{align}\label{eq5.21}
    \essosc_{\mathcal{Q}^\ominus_{\frac{cR}{4}}(\theta)}\leq \max\{(1-\eta)\omega,\hat{\gamma}\mbox{ Tail } [(u-\mu^\pm)_\pm;\tilde{\mathcal{Q}_0^\ominus}]\}=:\omega_1
\end{align}
where $\hat{\gamma}=\hat{\gamma}(\texttt{data})$.\vskip 1mm
Next, like the singular case, due to the presence of the tail term, it is unclear whether $\omega_1$ is controlled by $\omega$. That's why it is inevitable to estimate the tail term like in the previous singular case, that is, considering the positive truncated function, for instance, 
\begin{align*}
    \mbox{ Tail }[(u-\mu^+)_+;\tilde{\mathcal{Q}}_0^\ominus]&=\int_{-\mathcal{L}\theta(cR)^{sp}}^0\intb[\hn\setminus B]{R}\frac{(u-\mu^+)_+^{p-1}}{|x|_{\hn}^{Q+sp}}dxdt\\
    &\leq \gamma c^{sp}\omega+\gamma \int_{-\mathcal{L}\theta(cR)^{sp}}^0\intb[\hn\setminus B]{\mathcal{R}}\frac{u_+^{p-1}}{|x|_{\hn}^{Q+sp}}dxdt
\end{align*}
This inequality follows from the same process in the singular counterpart. Next, the last integral can be bounded by $\omega$ due to the assumption in (\ref{eq5.1}) and also from (\ref{eq5.21} we can write the following
\begin{align*}
    \mbox{ Tail } [(u-\mu^\pm)_\pm;\tilde{\mathcal{Q}_0^\ominus}]\leq \gamma\omega,\;\;\;\omega_1\leq \bar{\gamma}\omega
\end{align*}
for some $\gamma$ and $\gamma_1(>1)$ depending only on \texttt{data}. \vskip  1mm
Now, set $\theta_1=(\frac{1}{4}\omega_1)^{2-p}$ and $R_1=\lambda R$. We choose $\lambda$ in such a way that the following inclusion
\begin{align}\label{eq5.22}
    \mathcal{Q}^\ominus_{R_1}(\mathcal{L}\theta_1)\subset\mathcal{Q}^\ominus_{\frac{1}{4}cR}(\theta)\implies\lambda\leq \frac{1}{4}c\mathcal{L}^{-\frac{1}{sp}}(1-\eta)^\frac{p-2}{sp}
\end{align}
holds to initiate the induction. Next, combining (\ref{eq5.21}) and (\ref{eq5.22}), we have
\begin{align}
    \mathcal{Q}^\ominus_{R_1}(\mathcal{L}\theta_1)\subset\mathcal{Q}^\ominus_R(\mathcal{L}\theta)\;\;\;\mbox{ and }\;\;\;\essosc_{\mathcal{Q}^\ominus_{R_1}(\mathcal{L}\theta_1)}\leq \omega_1.
\end{align}
This is the initial step of the induction for the next stage. On the other hand, due to the $\mathcal{L}$, above results yields
\begin{align*}
    \essosc_{\mathcal{Q}^\ominus_{R_1}(\omega_1^{2-p})}\leq \omega_1.
\end{align*}
In this subsection, we determine $\mathcal{L}$ (see \ref{eq5.16}) in terms of \texttt{data} which was due in the previous section. But determination of $c$ is yet to be done. 
\subsection{The Induction}
We now use the method of induction. The setup for induction is similar to the previous setup in the singular case. Suppose for $i=1,2,3,...,j$, we construct
\begin{align*}
    \begin{cases}
        &R_0=R,\,R_1=\lambda R_{i-1}, \,\theta_1=(\frac{1}{4}\omega_1)^{2-p}, \, (1-\eta)\omega_{i-1}\leq \omega_i\leq \bar{\gamma}\omega_{i-1}\,\mbox{ with }\omega_0=\omega,\\
        &\omega_i=\max\{(1-\eta)\omega_{i-1},\,\hat{\gamma}\mbox{ Tail } [(u-\mu^\pm_{i-1})_\pm;\tilde{\mathcal{Q}}^\ominus_{i-1}]\}\;\mbox{ with } \tilde{\mathcal{Q}}^\ominus_i=\mathcal{Q}^\ominus_{R_i}(\mathcal{L}\theta_ic^{sp}),0], \\
        & \mathcal{Q}^\ominus=\mathcal{Q}^\ominus_{R_i}(\mathcal{L}\theta_i),\,\mathcal{Q}_i^\ominus\subset\mathcal{Q}_{i-1}^\ominus,\, \mu^+_i= \esssup_{\mathcal{Q}^\ominus_i}u,\, \mu^-=\essinf_{\mathcal{Q}^\ominus_i}u,\,\essosc_{\mathcal{Q}^\ominus_i}\leq \omega_i.
    \end{cases}
\end{align*}
Further, like in the Singular case, after repeating previous all arguments, for the sake of induction, we assume that the reduction of oscillation is valid up to the $j$-th step, that is, 
\begin{align*}
    \essosc_{\mathcal{Q}^\ominus_{\frac{1}{4}cR_j}(\theta_j)}u\leq (1-\eta)\omega_j,
\end{align*}
provided
\begin{align*}
    \hat{\gamma}\mbox{ Tail }[(u-\mu_j^\pm)_\pm;\tilde{\mathcal{Q}}_j^\ominus]\leq \omega_j.
\end{align*}
This tail estimate can be achieved after the same computation which lead to achieve the tail estimate in (\ref{eq4.12*}) as the process does not care about if $p<2$ or $p>2$ except one change that is replacement of $\omega_j^{p-2}$ with $\mathcal{L}(\frac{1}{4}\omega_j)^{2-p}$. In either case, we have
\begin{align}\label{eq5.24}
    \essosc_{\mathcal{Q}^\ominus_{\frac{1}{4}cR_j}(\theta_j)}u\leq \max\{(1-\eta)\omega_j,\hat{\gamma}\mbox{ Tail }[(u-\mu^\pm_j)_\pm;\tilde{\mathcal{Q}}^\ominus_j]\}:=\omega_{j+1}.
\end{align}
Next, set $\theta_{j+1}=(\frac{1}{4}\omega_{j+1})^{2-p}$ and $R_{j+1}=\lambda R_j$. It is obvious from above that $(1-\eta)\omega_j\leq \omega_{j+1}$. So the set inclusion
\begin{align}
    \mathcal{Q}_{R_{j+1}}^\ominus(\mathcal{L}\theta_{j+1})\subset \mathcal{Q}^\ominus_{\frac{1}{4}cR_j}(\theta_j),\;\;\;\mbox{ provided } \lambda\leq \frac{1}{4}c\mathcal{L}^{-\frac{1}{sp}}(1-\eta)^\frac{p-2}{sp}
\end{align}
which lead to, with the help of (\ref{eq5.24}), implies
\begin{align}
    \mathcal{Q}_{R_{j+1}}^\ominus(\mathcal{L}\theta_{j+1})\subset \mathcal{Q}_{R_{j}}^\ominus(\mathcal{L}\theta_{j})\;\;\;\mbox{ and } \essosc_{\mathcal{Q}_{R_{j+1}}^\ominus(\mathcal{L}\theta_{j+1})}u\leq \omega_{j+1},
\end{align}
this implies the following
\begin{align*}
   \essosc_{\mathcal{Q}_{R_{j+1}}^\ominus(\omega^{2-p}_{j+1})}u\leq \omega_{j+1}, 
\end{align*}
and thus we deduce the induction argument. The final choice of $c$ will be done in the next subsection. Also use of the estimation $\omega_i\leq \bar{\gamma}\omega_{i-1}$ will be done in forthcoming subsection. 
\subsection{Modulus of continuity for Degenerate Case}
By the of the induction in the last subsection, we have the oscillation estimate, for all $n\in\mathbb{N}$,
\begin{align}\label{eq5.27}
    \essosc_{\mathcal{Q}^\ominus_{R_n}(\omega_n^{2-p})}\leq \essosc_{\mathcal{Q}^\ominus_{\frac{1}{4}cR_n}(\theta_{n-1})}\leq \omega_n=\max\{(1-\eta)\omega_{n-1}, \hat{\gamma}\mbox{ Tail }[(u-\mu^\pm_{n-1})_\pm;\tilde{\mathcal{Q}}^\ominus_{n-1}]\}
\end{align}
with $\theta_n=(\frac{1}{4}\omega_n)^{2-p},\mathcal{Q}^\ominus_n=\mathcal{Q}^\ominus_{R_n}(\mathcal{L}\theta_n)$ such that $\mathcal{Q}^\ominus_n\subset\mathcal{Q}^\ominus_{n-1}$ for $n\in\mathbb{N}$. \vskip 1mm
Deriving an explicit modulus of continuity encoded in this oscillation estimate is similar to the subsection \ref{sub4.4}. Rest of this section we just only mention the main differences and briefly discuss the procedure. \vskip 1mm
The $(n-1)$th tail estimate with positive truncation
\begin{align}\label{eq5.28}
    \mbox{Tail }[(u-\mu^+)_+;\tilde{\mathcal{Q}}^\ominus_{n-1}]&\leq \gamma\omega_{n-1}\sum_{i=1}^{n-1}((1-\eta)^{1-p}c^{sp})^{n-i}+\gamma\int_{-\mathcal{L}\theta_{n-1}(cR_{n-1})^{sp}}^0\int_{\hn\setminus B_\mathcal{R}}\frac{u_+^{p-1}}{|x|_{\hn}^{Q+sp}}dxdt\nonumber\\
    &\leq\gamma\omega_{n-1}\sum_{i=1}^{n-1}((1-\eta)^{1-p}c^{sp})^{n-i}+\frac{\gamma}{\hat{\gamma}}\omega_{n-1}.
\end{align}
The first term on the right-hand side of (\ref{eq5.28}) can be dominated by $\frac{1}{2\hat{\gamma}}\omega_{n-1}$ is we choose 
\begin{align}\label{eq5.29}
    c\leq \Bigg(\frac{(1-\eta)^{p-1}}{2\gamma\hat{\gamma}}\Bigg)^\frac{1}{sp}
\end{align}
Hence, the right-hand side of (\ref{eq5.28}) is bounded by $\gamma\omega_{n-1}$ and by the construction of $\omega_n$ in (\ref{eq5.27}), we have
\begin{align}\label{eq5.30}
    \omega_n\leq \bar{\gamma}\omega_{n-1}
\end{align}
As we know by induction, $\theta_{n-1}R^{sp}_{n-1}\leq \theta R^{sp}$ then
\begin{align*}
    \mathcal{L}\theta_{n-1}(cR_{n-1})^{sp}\leq  \mathcal{L}c^{sp}\theta R^{sp}\leq \omega^{2-p}R^{sp}
\end{align*}
if we impose on $c$ that
\begin{align}\label{eq5.31}
    c\leq 2^\frac{4-2p}{sp}\mathcal{L}^\frac{-1}{sp}.
\end{align}
Under this restriction on $c$, we can further re-estimate the tail in (\ref{eq5.28}) in the new extended time interval that is 
\begin{align*}
   \mbox{Tail }[(u-\mu^+)_+;\tilde{\mathcal{Q}}^\ominus_{n-1}]&\leq \gamma\omega_{n-1}\sum_{i=1}^{n-1}((1-\eta)^{1-p}c^{sp})^{n-i}+\gamma\int_{-\omega^{2-p}R^{sp}}^0\int_{\hn\setminus B_\mathcal{R}}\frac{u_+^{p-1}}{|x|_{\hn}^{Q+sp}}dxdt.
\end{align*}
Next, like Singular case, iterating the previous estimate and joining it up with the oscillation estimate (\ref{eq5.27}), we have the following
\begin{align}\label{eq5.32}
  \essosc_{\mathcal{Q}^\ominus_{R_n}(\omega_n^{2-p})}u\leq (1-\eta)^n\omega+\gamma\int_{-\omega^{2-p}R^{sp}}^0\int_{\hn\setminus B_\mathcal{R}}\frac{u_+^{p-1}}{|x|_{\hn}^{Q+sp}}dxdt.
\end{align}
Analogous to (\ref{eq4.15}) we can obtain $\omega_n\leq \gamma^*\omega$ due to (\ref{eq5.1}). From previous estimate  Similar to the Singular section, fix some $r\in (0,R)$, there must be some $n\in\mathbb{N}$ such that
\begin{align*}
    \omega_{n+1}^{2-p}R^{sp}_{n+1}\leq \omega^{2-p}r^{sp}<\omega_n^{2-p}R^{sp}_n. 
\end{align*}
as the sequence $\{\omega_n^{2-p}R_n^{sp}\}$ tends to $0$. Next, the right-hand sid of the previous inequality implies that
\begin{align*}
    \sigma r<R_n\;\;\mbox{ and }\;\;\mathcal{Q}_{\sigma r}^\ominus(\omega^{2-p})\subset\mathcal{Q}_{R_n}^\ominus(\omega_n^{2-p})\;\;\mbox{ for }\sigma={\gamma^*}^\frac{2-p}{sp}
\end{align*}
whereas the left-hand side yields, with iterating (\ref{eq5.30}) the following
\begin{align*}
    (\bar{\gamma}^{2-p}\lambda^{sp})^{n+1}\omega^{2-p}R^{sp}\leq \omega_{n+1}^{2-p}R^{sp}_{n+1}\leq \omega^{2-p}r^{sp}.
\end{align*}
This implies
\begin{align*}
    \Big(\frac{r}{R}\Big)^{sp}\geq (\bar{\gamma}^{2-p}\lambda^{sp})^{n+1}\implies (1-\eta)^{n+1}\leq \Big(\frac{r}{R}\Big)^\mathfrak{b}
\end{align*}
where $$\mathfrak{b}:=\frac{sp\mbox{ ln }(1-\eta)}{\mbox{ ln}(\bar{\gamma}^{2-p}\lambda^{sp})}.$$
Moreover, iterating $(1-\eta)\omega_{n-1}\leq\omega_n$, we can have the following estimate
\begin{align*}
    \omega_n^{2-p}R_n^{sp}\leq ((1-\eta)^{2-p}\lambda^{sp})^n\omega^{2-p}R^{sp}\leq (1-\eta)^n\omega^{2-p}R^{sp},
\end{align*}
under the requirements
\begin{align}\label{eq5.33}
    \lambda<(1-\eta)^\frac{p-1}{sp}.
\end{align}
We choose the final $\lambda$ as the smaller between (\ref{eq5.22}) and (\ref{eq5.33}), once the choice of $c$ is done to be the smaller between (\ref{eq5.29}) and 
(\ref{eq5.31}).\vskip 1mm
Finally, collecting all the estimates in (\ref{eq5.32}), we have for all $r\in (0,R)$,
\begin{align}\label{eq5.34}
    \essosc_{\mathcal{Q}^\ominus_{\sigma r}(\omega^{2-p})}u\leq 2\omega(r/R)^\mathfrak{b}+\gamma\int_{-\omega^{2-p}R^{sp}}^0\intb[\hn\setminus B]{\mathcal{R}}\frac{|u|^{p-1}}{|x|^{Q+sp}_{\hn}}dxdt
\end{align}\vskip 10mm
Next, without loss of any generality, we may assume both oscillation estimations in (\ref{eq4.17}) and (\ref{eq5.34}) hold with the replacement of $R$ with some $\mathfrak{R}\in (r,R)$. Then, by taking $\mathfrak{R}=(rR)^\frac{1}{2}$, we obtain
\begin{align}\tag{GenCont}\label{GenCont}
   \essosc_{\mathcal{Q}^\ominus_{\sigma r}(\omega^{2-p})}u\leq 2\omega(r/R)^\mathfrak{b}+\gamma\int_{-\omega^{2-p}{(rR)}^\frac{sp}{2}}^0\intb[\hn\setminus B]{\mathcal{R}}\frac{|u|^{p-1}}{|x|^{Q+sp}_{\hn}}dxdt 
\end{align}
where $\mathfrak{b}:=\min\Big\{\frac{sp\mbox{ ln}(1-\eta)}{\mbox{ ln}((1-\eta)^{p-2}\lambda^{sp})},\frac{sp\mbox{ ln }(1-\eta)}{\mbox{ ln}(\bar{\gamma}^{2-p}\lambda^{sp})}\Big\}.$
\section{Proof of the Main Result}\label{S6}
Without any loss of generality we take $(x_0,t_0)=(0,0)$. Define 
\begin{align}\label{eq6.1}
    \omega:= 2\esssup_{\mathcal{Q}_\mathcal{R}^\ominus}|u|+\Bigg(\fint_{-\mathcal{R}^{sp}}^0\Big(\mathcal{R}^{sp} \intb[\hn\setminus B]{\mathcal{R}}\frac{|u|^{p-1}}{|x|_{\hn}^{Q+sp}}dx\Big)^{1+\varepsilon}dt\Bigg)^\frac{1}{1+\varepsilon}
\end{align}
where $\mathcal{Q}^\ominus_\mathcal{R}\subset \Omega_T$. It is obvious that by defining (\ref{eq6.1}), conditions in (\ref{eq4.1}) and (\ref{eq5.1}) remain intact. Therefore, by both sub and super critical cases that is considering (\ref{GenCont}) also holds that is
\begin{align}\label{eq6.2}
    \essosc_{\mathcal{Q}^\ominus_{\sigma r}(\omega^{2-p})}u\leq 2\omega(r/R)^\mathfrak{b}+\gamma\int_{-\omega^{2-p}{(rR)}^\frac{sp}{2}}^0\intb[\hn\setminus B]{\mathcal{R}}\frac{|u|^{p-1}}{|x|^{Q+sp}_{\hn}}dxdt 
\end{align}
for any $0<r<R$. Next, using H\"older inequality in the second term on the right hand side of (\ref{eq6.2}), we get
\begin{align}\label{eq6.3}
    \int_{-\omega^{2-p}{(rR)}^\frac{sp}{2}}^0\intb[\hn\setminus B]{\mathcal{R}}\frac{|u|^{p-1}}{|x|^{Q+sp}_{\hn}}dxdt&\leq (\omega^{2-p}{(rR)}^\frac{sp}{2})^\frac{\varepsilon}{1+\varepsilon}\Bigg(\int_{-\omega^{2-p}{(rR)}^\frac{sp}{2}}^0\Big(\intb[\hn\setminus B]{\mathcal{R}}\frac{|u|^{p-1}}{|x|^{Q+sp}_{\hn}}dx\Big)^{1+\varepsilon}dt\Bigg)^\frac{1}{1+\epsilon}\nonumber\\
    &\leq \frac{(\omega^{2-p}{(rR)}^\frac{sp}{2})^\frac{\varepsilon}{1+\varepsilon}}{(\omega^{2-p}R^{sp})^\frac{\varepsilon}{1+\varepsilon}}\Bigg(\fint_{-\mathcal{R}^{sp}}^0\Big(\mathcal{R}^{sp} \intb[\hn\setminus B]{\mathcal{R}}\frac{|u|^{p-1}}{|x|_{\hn}^{Q+sp}}dx\Big)^{1+\varepsilon}dt\Bigg)^\frac{1}{1+\varepsilon}\nonumber\\
    &\leq \omega\Big(\frac{r}{R}\Big)^\frac{\varepsilon sp}{2(1+\varepsilon)} 
\end{align}
Combining (\ref{eq6.2}) and (\ref{eq6.3}), we have the desired result.

\end{document}